\documentclass{article}
\usepackage{amsmath,amsfonts,amssymb}
\usepackage[colorlinks,citecolor=blue,urlcolor=blue]{hyperref}
\RequirePackage{graphicx}

\newtheorem{theorem}{Theorem}

\newtheorem{definition}[theorem]{Definition}

\newtheorem{lemma}[theorem]{Lemma}

\newtheorem{proposition}[theorem]{Proposition}
\newtheorem{remark}[theorem]{Remark}

\newenvironment{proof}[1][Proof]{\noindent\textbf{#1.} }{\ \rule{0.5em}{0.5em}}

\begin{document}

\title{\textbf{Existence of relaxed optimal control for $G$-neutral stochastic functional differential equations with uncontrolled diffusion}}

\author{\textbf{Nabil Elgroud$^{a}$, Hacene Boutabia$^{a,*}$, Amel Redjil$^{a,\dagger}$, Omar Kebiri$^{b}$ }\\
$^{a}$Department of Mathematics, Badji MokhtarUniversity\\ 
Annaba, 23000 Annaba, Algeria.\\
$^{b}$Institute of Mathematics, Brandenburgische Technische \\
University Cottbus-Senftenberg, 03046 Cottbus, Germany.
}
\date{}
\maketitle

\begin{abstract}
In this paper, we study under refined Lipchitz hypothesis, the 
question of existence and uniqueness of solution of controlled neutral stochastic functional differential equations driven by $G$-Brownian motion 
($G$-NSFDEs in short). An existence of a relaxed optimal control where 
the neutral and diffusion terms do not depend on the control variable 
was the main result of the article. The latter is done by using 
tightness techniques and the weak convergence techniques for each 
probability measure in the set of all possible probabilities of our 
dynamic. A motivation of our work is presented and a Numerical 
analysis for the uncontrolled $G$-NSFDE is given.
\end{abstract}


\section{Introduction}\label{int1}
Due to the important ambiguous concepts in the study of optimal control
problems in finance under the principles of uncertainty, it appears in
different typical fields that contain incomplete or inaccurate parameters,
especially financial crises and risks resulting from dark fluctuations and
their impact on the movement of asset prices and liquidity in the markets.

The concepts of uncertainty in fluctuations were studied by \cite{peng2007, peng2010}, who established a type of non-linear expectation theory or expectancy theory within the framework of $G$-Brownian motion, and  \cite{denis2006, denis2011}, did that through the capacity theory, and then relied on the $G$-Brownian movement under $G$-expectation to create $G-$stochastic calculus and this is what led both to prove the existence and uniqueness of the stochastic differential equations driven by the $G$-Brownian motion by \cite{gao2009, peng2007}. In addition, \cite{faizullah2016, faizullah2017} studied the existence and uniqueness of neutral stochastic functional differential equations within the framework of the $G$-Brownian motion ($G$-\text{NSFDEs in short}), is given by 
\begin{center}
\begin{equation}
\left\{ 
\begin{array}{l}
d\left[ X\left( t\right) -Q\left( t,X_{t}\right) \right] =b\left(
t,X_{t}\right) dt+\gamma \left( t,X_{t}\right) d\left\langle B\right\rangle
_{t}\ +\sigma \left( t,X_{t}\right) dB_{t}, \quad t\in \left[ 0,T\right] \\ 
X_{0}=\eta ,
\end{array}
\right.    \label{UCOGNSDE}
\end{equation}
\end{center}
where, $\eta \in BC\left( \left[ -\tau ,0\right] ;\mathbb{R}\right) $, and $%
\tau \geq 0,$ $X_{t}=\left\{ X\left( t+\theta \right) :-\tau \leq \theta
\leq 0\right\} $, $\left( B_{t},t\geq 0\right) $ is a one-dimensional $G$-
Brownian motion defined on some space of sublinear expectation $\left(
\Omega ,\mathcal{H},\widehat{E},\mathbb{F}^{\mathcal{P}}\right) ,$ with a
universal filtration $\mathbb{F}^{\mathcal{P}}=\left\{ \widehat{\mathcal{F}}%
_{t}^{\mathcal{P}}\right\} _{t\geq 0}$, and $\left\{ \left\langle
B\right\rangle _{t},\text{ }t\geq 0\right\} $ is the quadratic variation
process of $G$-Brownian motion, $Q$, $b$, $\gamma $, and $\sigma $ are
deterministic functions on $\left[ 0,T\right] \times BC\left( \left[ -\tau ,0%
\right] ;\mathbb{R}\right) $. With what the $G$-expectation permits%
\begin{equation*}
\widehat{E}\left[ .\right] =\underset{\mathbb{P\in }\mathcal{P}}{\sup }E^{%
\mathbb{P}}\left[ .\right],
\end{equation*}%
where $E^{\mathbb{P}}$ is ordinary expectations, and $\mathcal{P}$ is a
tight family of possibly mutually singular probability measures. For more
details see \cite{denis2006, denis2011}. Recently, \cite{biagini2018, hu2018, hu2014, redjil2018} considered an optimal control problem with the uncertainty of $G$-Brownian motion and its quadratic variation $\left\langle B\right\rangle.$ In this paper we consider the following $G$-NFSDE
\begin{center}
\begin{equation}
\left\{ 
\begin{array}{l}
d\left[ X^{u}\left( t\right) -Q\left( t,X_{t}^{u}\right) \right] =b\left(
t,X_{t}^{u},u\left( t\right) \right) dt+\gamma \left( t,X_{t}^{u},u\left(
t\right) \right) d\left\langle B\right\rangle _{t}\ +\sigma \left(
t,X_{t}^{u}\right) dB_{t} \\ 
X_{0}^{u}=\eta ,t\in \left[ 0,T\right]%
\end{array}%
\right.  \label{1.2}
\end{equation}%
\end{center}
where $u\left( .\right) \in \mathbb{A}$ stands for the control variable for
each $t\in \left[ 0,T\right] ,$ and $\mathbb{A}$ is a compact polish space
of $\mathbb{R}.$ Let $\mathcal{P}\left( \mathbb{A}\right) $ denote the space
of probability measures on $\mathcal{B}(\mathbb{A})$, the $\sigma$-algebra
of Borel subsets of the set $\mathbb{A}$ of values taken by the strict
control. The set $\mathcal{U=U}\left( \left[ 0,T\right] \right)$ is a set
of strict controls. The case of a controlled SDE driven by a classical Brownian motion has been treated by different authors, see e.g \cite{ahmed2014, bahlali2006, wei2015}. In this paper, we study under the concepts presented in \cite{peng2007, peng2010} the existence of a relaxed optimal control that minimize the cost functional:
\begin{equation}
\widehat{E}\left[ \int_{0}^{T}\mathcal{L}(t,X_{t}^{u},u\left( t\right)
)\,dt+\Psi (X_{T}^{u})\right],  \label{1.3}
\end{equation}%
The proof is based on the tightness arguments of the distribution of the control problem.

\textbf{Motivation:} To motivate our work let consider a Brownian particle
moving in an unbounded medium. Let $X(t)$ be the position and $Y(t)$ the
velocity of the particle at time $t.$ So The dynamic is represented by

\begin{equation}
X^{\prime }(t)=Y(t)\qquad \qquad mdY(t)=b(t)dt+\sigma d\xi _{t},  \label{1.4}
\end{equation}%
where $m$ is the mass of the particle and $\sigma d\xi _{t}$ is the noise
part of the medium on the particle. According to Boussinesq representation in 
\cite{boussinesq1885}, $b(t)=-hY(t)-Y^{\prime }(t)\sigma \sqrt{\frac{%
hm_{1}^{3}}{\pi }}\int_{0}^{\infty }Y^{\prime 2}\left( s\right) ds,$ which
represents the systematic action of the medium on the particle, where $%
-hY(t)$ is the Stokes friction force at time $t$ and $m$ the apparent
additional mass which is half the mass of the material of the medium ousted
by the body. The $\int_{0}^{\infty }Y^{\prime 2}\left( s\right) ds$ is the
viscous hydrodynamic aftereffect. These models represent a NSFDE in the
classical case.

In reality, it is difficult to estimate exactly the noise parameter $\sigma $, and what we can have as information is only a range interval $[\sigma
_{min},\sigma _{max}]$ where $\sigma$ belongs, and so, the question is to
study the worse-case scenario, which is difficult to analyse it by direct
methods. the worst scenario system can me transformed to a $G$-NSFDE, and if
we want to control the dynamic of the particle subject to some constrain,
this will leads to a stochastic optimal control driven by a $G$-NSFDE.

The rest of the paper is formed as follows. In section $\ref{prl2}$, we introduce some preliminaries which will be used to establish our result. In section $\ref{fp3}$, is related to three topics, first, we are concentrated to introduce the Problem of $G$-NSFDEs relaxed control, secondly, we prove the existence and uniqueness of solution of $G$-NSFDEs with uncontrolled diffusion, we established the existence of a minimizer of the cost functional in third. Finally, we study the approximation of the relaxed control and we prove the existence of relaxed control. The last section is devoted to some numerical analysis.

\section{Preliminaries}\label{prl2}

The main purpose of this section is to introduce some basic notions and
results in $G$-stochastic calculus that are used in the subsequent sections.
More details can be found in \cite{denis2006, denis2011, peng2007, peng2010, soner2011m, soner2011}.\\
We set $\Omega :=\{\omega \in C(\left[ 0,T\right] ,\mathbb{R}):\omega
(0)=0\}$, the space of real valued continuous functions on $\left[ 0,T%
\right]$ such that $\omega (0)=0,$\ equipped with the following distance

\begin{equation*}
d\left( w^{1},w^{2}\right) :=\sum_{N=1}^{\infty }2^{-N}\left( \left( 
\underset{0\leq t\leq N}{\max }\left\vert w_{t}^{1}-w_{t}^{2}\right\vert
\right) \wedge 1\right),
\end{equation*}
$\Omega _{t}:=\{w_{.\wedge t}:w\in \Omega \}$, $B_{t}\left( w\right) =w_{t},
t\geq 0$ the canonical process on $\Omega $ and let $\mathbb{F}:=(\mathcal{F}%
_{t}\,)_{\,t\geq 0}$ be the natural filtration generated by $%
(B_{t}\,)_{\,t\geq 0}$. Moreover, we set, for each $t\in \left[
0,\infty\right) $%
\begin{eqnarray*}
\mathcal{F}_{t+} &:&=\cap _{s>t}\mathcal{F}_{s}, \\
\mathbb{F^{+}} &:&=(\mathcal{F}_{t^{+}}\,)_{\,t\geq 0}, \\
\mathcal{F}_{t}^{\mathbb{P}} &:&=\mathcal{F}_{t+}\vee \mathcal{N}^{\mathbb{P}%
}(\mathcal{F}_{t+}), \\
\widehat{\mathcal{F}}_{t}^{\mathbb{P}} &:&=\mathcal{F}_{t+}\vee \mathcal{N}^{%
\mathbb{P}}(\mathcal{F}_{\infty }),
\end{eqnarray*}

where $\mathcal{N}^{\mathbb{P}}(\mathcal{G})$ is a $\mathbb{P}$-negligible
set on a $\sigma $-algebra $\mathcal{G}$ given by%
\begin{equation*}
\mathcal{N}^{\mathbb{P}}(\mathcal{G}):=\{D\subset \Omega :\text{there\ exists%
}\ \widetilde{D}\in \mathcal{G}\ \ \text{such\ that}\ D\subset \widetilde{D}%
\ \text{\ and}\ \mathbb{P}[\widetilde{D}]=0\},
\end{equation*}

where $\mathbb{P}$ is a probability measure on the Borel $\sigma $-algebra $%
\mathcal{B}(\Omega )$ of $\Omega $. Consider the following spaces: for $
0\leq t\leq T$%
\begin{eqnarray*}
Lip(\Omega _{t}) &:&=\left\{ \varphi (B_{t_{1}},...,B_{t_{n}}):\varphi \in
C_{b,Lip}(\mathbb{R}^{n})\text{ and }t_{1},t_{2},...,t_{n}\in \left[ 0,t%
\right] \right\}, \\
Lip(\Omega ) &:&=\underset{n\in \mathbb{N}}{\cup }Lip(\Omega _{n}),
\end{eqnarray*}
where $C_{b,Lip}(\mathbb{R}^{n})$ is the space of bounded and Lipschitz on $%
\mathbb{R}^{n}.$ Let $T>0$ be a fixed time.

\cite{peng2007} has constructed the $G$-expectation $\widehat{E}$ 
$:\mathcal{H}:=Lip(\Omega _{T})\longrightarrow \mathbb{R}$ which is a
consistent sublinear expectation on the lattice $\mathcal{H}$ of real
functions $i.e.$ it satisfies:

\begin{enumerate}
\item Sub-additivity: $\widehat{E}[X+Y]\leq \widehat{E}[X]+\widehat{E}[Y],$
for all $X,Y\in \mathcal{H},\quad $

\item Monotonicity: $X\geq Y\Rightarrow \widehat{E}[X]\geq \widehat{E}[Y],$
for all $X,Y\in \mathcal{H},$

\item Constant preserving: $\widehat{E}[c]=c,$ for all $c\in \mathbb{R},$

\item Positive homogeneity: $\quad \widehat{E}[\lambda X]=\lambda \widehat{E}%
[X],$ for all $\lambda \geq 0,\,X\in \mathcal{H},$
\end{enumerate}
The triple $\left( \Omega ,\mathcal{H},\widehat{E}\right )$ is said to be
sub-linear expectation space, if \textbf{1} and \textbf{2} are only
satisfied. Moreover, $\widehat{E}\left[ .\right] $ is called a nonlinear
expectation and the triple $\left( \Omega ,\mathcal{H},\widehat{E}\right)$
is called a nonlinear expectation space.

we assume that, if $Y=(Y_{1},...,Y_n),Y_{i}$ $\in \mathcal{H}$, then $%
\varphi(Y_{1},...,Y_n)\in \mathcal{H}$ for all $\varphi $ $\in C_{b,Lip}(%
\mathbb{R}^{n}).$

\begin{definition}
A random vector $Y=(Y_{1},...,Y_{n})$ is said to be independent from another
random vector $X=(X_{1},...,X_{m})$ under $\widehat{E}$ if for any $%
\varphi\in C_{b,Lip}(\mathbb{R}^{n+m})$%
\begin{equation*}
\widehat{E}\left[ \varphi \left( X,Y\right) \right] =\widehat{E}\left[ 
\widehat{E}\left[ \varphi \left( x,Y\right) \right] _{x=X}\right].
\end{equation*}
\end{definition}

\begin{definition}
A process $X$ on $\left( \Omega ,\mathcal{H},\widehat{E}\right) $ is said to
be $G$-normally distributed under the $G$-expectation $\widehat{E}[\cdot ]$
if \ for any $\varphi $ $\in C_{b,Lip}(\mathbb{R})$\ the function%
\begin{equation*}
u\left( t,x\right) :=\widehat{E}\left[ \varphi \left( x+\sqrt{t}X\right) %
\right] ,\left( t,x\right) \in \left[ 0,T\right] \times \mathbb{R},
\end{equation*}%
is the unique viscosity solution of the parabolic equation%
\begin{equation*}
\left\{ 
\begin{array}{c}
\frac{\partial u}{\partial t}=G\left( u_{xx}\right) \\ 
u\left( 0,x\right) =\varphi \left( x\right)
\end{array}%
\right.
\end{equation*}%
where t\textit{he nonlinear function }$G$\textit{\ is defined by }$G(a):=$%
\textit{\ }$\frac{1}{2}\widehat{E}\left[ aX^{2}\right] =\frac{1}{2}\left( 
\overline{\sigma }^{2}a^{+}-\underline{\sigma }^{2}a^{-}\right) ,$\textit{\ }%
$a\in \mathbb{R},$\textit{with }$\overline{\sigma }^{2}:=$\textit{\ }$%
\widehat{E}\left[ X^{2}\right] $\textit{, }$\underline{\sigma }^{2}:=-%
\widehat{E}\left[ -X^{2}\right] ,$\textit{\ }$a^{+}$\textit{=}$\max \left\{
0,a\right\} $\textit{\ and }$a^{-}$\textit{= -}$\min \left\{ 0,a\right\} $%
\textit{. This }$G$-\textit{normal distribution is denoted by }$\mathcal{N}%
(0,$\textit{\ }$\left[ \underline{\sigma }^{2},\text{ }\overline{\sigma }^{2}%
\right] )$\textit{. }
\end{definition}

\begin{definition}
$\left( \text{\textbf{G-Brownian Motion}}\right) $The canonical process $%
\left( B_{t}\right) _{t\geq 0}$ on $\left( \Omega ,\text{ }\mathcal{H},\text{%
}\widehat{E}\right) $ is called a $G$-Brownian motion if the following
properties are satisfied:

\begin{itemize}

\item
$B_{0}=0.$

\item
\textit{For each} $t,s\geq 0$ \textit{the increment%
} $B_{t+s}-B_{t}$ \ is $\mathcal{N}(0,$ $\left[ s\underline{\sigma }^{2},%
\text{ }s\overline{\sigma }^{2}\right] )$-distributed$.$

\item
$B_{t_{1}},B_{t_{2}},...,B_{t_{n}}$ is \textit{%
independent of} $B_{t},$ \textit{for\ }$n\geq 1$ \textit{and} $%
t_{1},t_{2},...,t_{n}\in \left[ 0,t\right] .$
\end{itemize}
\end{definition}

For $p\geq 1,$ we denote by $L_{G}^{p}(\Omega _{T})$ the completion of $%
Lip(\Omega _{T})$ under the natural norm%
\begin{equation*}
\Vert X\Vert _{L_{G}^{p}(\Omega _{T})}^{p}:=\widehat{E}[|X|^{p}],
\end{equation*}

and define the space $M_{G}^{0,p}(0,T)$ of $\mathbb{F}$-progressively
measurable, $\mathbb{R}$-valued simple processes of the form%
\begin{equation*}
\eta (t)=\eta (t,w)=\sum_{i=0}^{n-1}\xi _{t_{i}}\left( w\right) \mathbb{
\nparallel }_{[t_{i},t_{i+1})}(t)\text{ ,}
\end{equation*}

where $\left\{ t_{0},\cdots ,t_{n}\right\} $ is a subdivision of $\left[ 0,T%
\right]$. Denote by $M_{G}^{p}(0,T)$ the closure of $M_{G}^{0,p}(0,T)$ with
respect to the norm%
\begin{equation*}
\Vert \eta \Vert _{M_{G}^{p}(0,T)}^{p}:=\hat{\mathbb{E}}[\int_{0}^{T}|\eta
(t)|^{p}ds].
\end{equation*}

Note that $M_{G}^{q}(0,T)\subset M_{G}^{p}(0,T)$ if $1\leq p<q$. For each $%
t\geq 0$, let $L^{0}(\Omega _{t})$ be the set of $F_{t}$-measurable
functions. We set%
\begin{equation*}
Lip(\Omega _{t}):=Lip(\Omega )\cap L^{0}(\Omega _{t}),\quad
L_{G}^{p}(\Omega_{t}):=L_{G}^{p}(\Omega )\cap L^{0}(\Omega _{t}).
\end{equation*}

For each $\eta \in M_{G}^{0,2}(0,T)$, the related It\^{o} integral of $%
\left( B_{t}\right) _{t\geq 0}$ is defined by%
\begin{equation*}
I(\eta )=\int_{0}^{T}\eta \left( s\right) dB_{s}:=\sum_{j=0}^{N-1}\eta
_{j}(B_{t_{j+1}}-B_{t_{j}}),
\end{equation*}

where the mapping $I:\,M_{G}^{0,2}(0,T)\rightarrow L_{G}^{2}(\Omega _{T})$is
continuously extended to $M_{G}^{2}(0,T).$ The quadratic variation process $%
\langle B\rangle _{t}$ of $\left( B_{t}\right) _{t\geq 0},$ defined by%
\begin{equation}
\langle B\rangle _{t}:=B_{t}^{2}-2\int_{0}^{t}B_{s}dB_{s}  \label{2.1}
\end{equation}%
\ \ 

For each $\eta \in M_{G}^{0,1}(0,T)$, Let the mapping $\mathcal{J}%
_{0,T}\left( \eta \right) :M_{G}^{0,1}(0,T)\mapsto \mathbb{L}%
_{G}^{1}(\Omega_{T})$ given by:

\begin{equation*}
\mathcal{J}_{0,T}\left( \eta \right) =\int_{0}^{T}\eta \left( t\right)
d\langle B\rangle _{t}:=\sum_{j=0}^{N-1}\xi _{j}(\langle B\rangle
_{t_{j+1}}-\langle B\rangle _{t_{j}}).
\end{equation*}

Then $\mathcal{J}_{0,T}\left( \eta \right) $ can be extended continuously to 
\begin{equation*}
\mathcal{J}_{0,T}\left( \eta \right) :M_{G}^{1}(0,T)\rightarrow \mathbb{L}%
_{G}^{1}(\Omega _{T}).
\end{equation*}

\begin{lemma}
(\cite{peng2010}) We have for each $p\geq 1$

\begin{equation*}
\widehat{E}\left[ \int_{0}^{T}\eta (t)d\langle B\rangle _{t}\right] \leq 
\overline{\sigma }^{2}\widehat{E}\left[ \int_{0}^{T}\left\vert \eta
(t)\right\vert dt\right] ,\text{ for each }\eta \in M_{G}^{1}(0,T).
\end{equation*}%
\newline
\begin{equation*}
\widehat{E}\left[ \left( \int_{0}^{T}\eta (t)dB_{t}\right) ^{2}\right] =%
\widehat{E}\left[ \int_{0}^{T}\eta ^{2}(t)d\langle B\rangle _{t}\right] ,%
\text{ for each }\eta \in M_{G}^{2}(0,T)\text{ }\left( \text{isometry}%
\right) .
\end{equation*}%
\newline
\begin{equation*}
\widehat{E}\left[ \int_{0}^{T}\left\vert \eta (t)\right\vert ^{p}dt\right]
\leq \int_{0}^{T}\widehat{E}\left[ \left\vert \eta (t)\right\vert ^{p}\right]
dt,\text{ for each }\eta \in M_{G}^{p}(0,T).
\end{equation*}
\end{lemma}

\begin{proposition}
(\cite{denis2006}) \ For each $\xi \in {\mathbb{L}}_{G}^{1}(\Omega ).$There
exists a weakly compact family of probability measures $\mathcal{P}$ on $%
(\Omega ,$ $\mathcal{B}(\Omega ))$ such that%
\begin{equation*}
\widehat{E}[\xi ]=\sup_{\mathbb{P}\in \mathcal{P}}E^{\mathbb{P}}[\xi ].
\end{equation*}%
Then, we define the associated regular choquet capacity related to $\mathbb{P
}$: 
\begin{equation*}
c(C):=\sup_{\mathbb{P}\in \mathcal{P}}\mathbb{P}(C),\quad C\in \mathcal{B}%
(\Omega ).
\end{equation*}
\end{proposition}

\begin{definition}
\textit{A set }$C\in B(\Omega )$\textit{\ is polar if }$c(C)$\textit{\ }$=0$%
\textit{\ or equivalently if }$\,\mathbb{P}(C)=0$\textit{\ for all }$\mathbb{%
P}\in \mathcal{P}.$ A property holds quasi surely ( $q.s.$ in short) if it
holds outside a polar set.
\end{definition}

Let define $\mathcal{N}_{\mathcal{P}}$ the $\mathcal{P}$-polar sets, as
follow%
\begin{equation*}
\mathcal{N}_{\mathcal{P}}:=\bigcap_{\mathbb{P}\in \mathcal{P}}\mathcal{N}^{%
\mathbb{P}}(\mathcal{F}_{\infty }).
\end{equation*}

We must use the following universal filtration $\mathbb{F}^{\mathcal{P}}$
for the possibly mutually singular probability measures $\mathbb{P},\mathbb{P}\in \mathcal{P}$ in \cite{soner2011}.%
\begin{eqnarray*}
\mathbb{F}^{\mathcal{P}} &:&=\{\widehat{\mathcal{F}}_{t}^{\mathcal{P}%
}\}_{t\geq 0},\quad \quad \\
\quad \widehat{\mathcal{F}}_{t}^{\mathcal{P}} &:&=\bigcap_{\mathbb{P}\in 
\mathcal{P}}(\mathcal{F}_{t}^{\mathbb{P}}\vee \mathcal{N}_{\mathcal{P}%
})\quad \text{for}\quad t\geq 0.
\end{eqnarray*}

In view of the dual formulation of the $G$-expectation, we end this section
by the following Burkholder-Davis-Gundy-type estimates, formulated in one
dimension.

\begin{proposition}
(\cite{gao2009})
\begin{itemize}

\item 
\textit{For each }$p\geq 2$\textit{\ and }$\eta\in
M_{G}^{p}(0,T)$\textit{, then there exists some constant }$C_{p}$\textit{\
depending only on }$p$\textit{\ and }$T$\textit{\ such that} 
\begin{equation*}
\widehat{E}\left[ \sup_{s\leq u\leq t}\left\vert \int_{s}^{u}\eta
_{r}dB_{r}\right\vert ^{p}\right] \leq C_{p}|t-s|^{\frac{p}{2}-1}\int_{s}^{t}
\widehat{E}[|\eta _{r}|^{p}|]dr.
\end{equation*}

\item
\textit{For each }$p\geq 1$\textit{\ and }$\eta
\in M_{G}^{p}(0,T),$\textit{\ then there exists a positive constant }$\bar{%
\sigma}$\textit{\ such that }$\frac{d\langle B\rangle _{t}}{dt}\leq \bar{%
\sigma}$\textit{\ }$q.s.$, \textit{\ we have} 
\begin{equation*}
\widehat{E}\left[ \sup_{s\leq u\leq t}\left\vert \int_{s}^{u}\eta
_{r}d\langle B\rangle _{r}\right\vert ^{p}\right] \leq \bar{\sigma}%
^{p}|t-s|^{p-1}\int_{s}^{t}\hat{\mathbb{E}}[|\eta _{r}|^{p}]dr.
\end{equation*}
\end{itemize}
\end{proposition}

\section{\protect\bigskip Formulation of the problem}\label{fp3}
We study the existence of optimal control problem for $G$-NSFDEs, given the
following integral equation%
\begin{eqnarray}
X\left( t\right) &=&\eta \left( 0\right) +Q\left( t,X_{t}\right) -Q\left(
0,\eta \right) +\int_{0}^{t}b\left( s,X_{s},u\left( s\right) \right) ds\ \ 
\notag \\
&&+\int_{0}^{t}\gamma \left( s,X_{s},u\left( s\right) \right) d\left\langle
B\right\rangle _{s}\ +\int_{0}^{t}\sigma \left( s,X_{s}\right) dB_{s},t\in 
\left[ 0,T\right]  \label{3.1}
\end{eqnarray}%
$\ \ \ \ \ \ \ \ \ \ \ \ \ \ \ \ \ \ \ \ \ \ \ \ \ \ $

with random initial data%
\begin{equation*}
\eta =\left\{ \eta \left( \theta \right) \right\} _{-\tau \leq \theta \leq
0}\in BC\left( \left[ -\tau ,0\right] ;\mathbb{R}\right),
\end{equation*}

with $BC\left( \left[ -\tau ,0\right] ;\mathbb{R}\right) $\ is a space of $%
\mathbb{R}$-valued functions defined on $\left[ -\tau ,0\right] $ and $\tau
>0,$ where $X_{t}=\left\{ X\left( t+\theta \right) :-\tau \leq \theta \leq
0\right\} ,$ and $u\left( t\right) $ $\mathbb{\in A}$ is called a strict
control variable for each $t\in \left[ 0,T\right] .$ Let the space%
\begin{equation*}
\widetilde{\mathcal{H}}_{T}:=\left\{ X=\left( X\left( t\right) \right)
_{t\in \left[ 0,T\right] },\text{ }\mathbb{F}^{\mathcal{P}}-\text{adapted
such that:}\int_{0}^{T}\widehat{E}\left[ \left\vert X\left( s\right)
\right\vert ^{2}\right] ds<\infty \right\} ,
\end{equation*}%
equipped with the norms $N_{C}\left( X\right) :=\left( \int_{0}^{T}\exp
\left( -2Cs\right) \widehat{E}\left( \left\vert X\left( s\right) \right\vert
^{2}\right) ds\right) ^{\frac{1}{2}},$ where $C\geq 0.$ Since 
\begin{equation*}
\exp \left( -2CT\right) N_{0}\left( X\right) \leq N_{C}\left( X\right) \leq
N_{0}\left( X\right) ,
\end{equation*}%
then these norms are equivalent. Moreover, the functions%
\begin{equation*}
Q,\sigma :\left[ 0,T\right] \times BC\left( \left[ -\tau ,0\right] ;\mathbb{R%
}\right) \times \Omega \rightarrow \mathbb{R},
\end{equation*}%
\begin{equation*}
b,\gamma :\left[ 0,T\right] \times BC\left( \left[ -\tau ,0\right] ;\mathbb{R%
}\right) \times \mathbb{A}\times \Omega \rightarrow \mathbb{R},
\end{equation*}%
are measurable, the random variable $Q\left( 0,0\right) \in L_{G}^{2}\left(
\Omega _{T}\right) $ as well as $Q\left( .,x\right) ,$ $\sigma \left(
.,x\right) ,$ $b\left( .,x,u\left( .\right) \right) ,$ $\gamma \left(
.,x,u\left( .\right) \right) \in $ $\widetilde{\mathcal{H}}_{T}$ for each $%
x\in BC\left( \left[ -\tau ,0\right] ;\mathbb{R}\right) $ and for each
strict control $u.$

\subsection{Problem of G-NSFDE relaxed control.}

In the absence of convexity assumptions, the strict control problem may not
have an optimal solution because $\mathbb{A}$ is too small to contain a
minimizer. Then the space of strict controls must be injected into a wider
space that has good properties of compactness and convexity. The set $%
\mathbb{A}$ is a compact Polish space, and $\mathcal{P}\left( \mathbb{A}%
\right) $ be the space of probability measures on $\mathbb{A}$, endowed with
its Borel $\sigma $-algebra $\mathcal{B}(\mathbb{A})$, ( \text{For
more details see} \cite{redjil2018}).

Next, we introduce the class of relaxed stochastic controls on $(\Omega , 
\mathcal{H},\widehat{E})$.

\begin{definition}
A relaxed stochastic control on $(\Omega ,\mathcal{H},\widehat{E})$ is an $%
\mathbb{F}^{\mathcal{P}}$-progressively measurable random measure of the
form $q(\omega ,dt,d\xi )=\mu _{t}(\omega ,d\xi )dt$ such that%
\begin{equation}
\begin{array}{l}
X\left( t\right) =\eta \left( 0\right) +Q\left( t,X_{t}\right) -Q\left(
0,\eta \right) +\int_{0}^{t}\int_{\mathbb{A}}b\left( s,X_{s},\xi \right) \mu
_{s}(d\xi )ds\ \ \  \\ 
\text{ \ \ \ \ \ \ \ }+\int_{0}^{t}\int_{\mathbb{A}}\gamma \left(
s,X_{s},\xi \right) \mu _{s}(d\xi )d\left\langle B\right\rangle _{s}\
+\int_{0}^{t}\sigma \left( s,X_{s}\right) dB_{s},\text{ }t\in \left[ 0,T%
\right].%
\end{array}%
\   \label{3.2}
\end{equation}
\end{definition}

Note that each strict control can be considered as a relaxed control via the
mapping%
\begin{equation}
\Phi (u)(dt,d\xi )=\delta _{u(t)}\left( d\xi \right).dt,  \label{3.3}
\end{equation}

where $\delta _{u\left( t\right) }$ is a Dirac measure charging $u\left(
t\right) $ for each $t.$

\begin{remark}
We mean by ``the process $q(\omega ,dt,d\xi )$ is $\mathbb{F}^{\mathcal{P}}$-
progressively measurable'' that for every $C\in \mathcal{B}(\mathbb{A})$\ and
for every $t\in \lbrack 0,T]$, the mapping $(s,\omega )\mapsto
\mu_{s}(\omega ,C)$ is $B([0,t])\otimes\widehat{\mathcal{F}}_{t}^{\mathcal{P}
} $-measurable. In particular, the process $(\mu _{t}(C))_{t\in \lbrack
0,T]} $ is adapted to $\mathbb{F}^{\mathcal{P}}$.
\end{remark}

We denote by $\mathcal{R}$ the class of relaxed stochastic controls.

\subsection{\protect\bigskip\ Existence and uniqueness of solution for
G-NSFDE}

In order to consider control problem $(\ref{3.1})$, we first study the
question of existence and uniqueness of solution to the following equation%
\begin{equation}
\ 
\begin{array}{l}
X\left( t\right) =\eta \left( 0\right) +Q\left( t,X_{t}\right) -Q\left(
0,\eta \right) +\int_{0}^{t}\int_{\mathbb{A}}b\left( s,X_{s},\xi \right) \mu
_{s}(d\xi )ds \\ 
\text{ \ \ \ \ \ \ \ }+\int_{0}^{t}\int_{\mathbb{A}}\gamma \left(
s,X_{s},\xi \right) \mu _{s}(d\xi )d\left\langle B\right\rangle _{s}\
+\int_{0}^{t}\sigma \left( s,X_{s}\right) dB_{s},\text{ }t\in \left[ 0,T%
\right].%
\end{array}
\label{3.4}
\end{equation}

where $\mu _{t}(d\xi )=\delta _{u\left( t\right) }(d\xi ).$

To guarantee existence and uniqueness of the solution of the equation $(\ref{3.1})$, we need the following assumptions:

\begin{description}
\item[$\left( A_{1}\right) $] There exists $K_{1}>0$ such that 
\begin{equation*}
\left\vert H\left( t,x,u\right) -H\left( t,y,u\right) \right\vert \leq
K_{1}\left\vert x\left( 0\right) -y\left( 0\right) \right\vert , \ \ \ \ \ \ \ 
\end{equation*}
\end{description}

uniformly with respect to $(t,\omega )$ for each $x,y\in BC\left( \left[ -\tau
,0\right] ;\mathbb{R}\right) $, where $H=b,\gamma ,\sigma .$

\begin{description}
\item[$\left( A_{2}\right) $] There exists $0<k_{0}<\frac{1}{4}$ such that 
\begin{equation}
\left\vert Q\left( t,x\right) -Q\left( t,y\right) \right\vert \leq
k_{0}\left\vert x\left( 0\right) -y\left( 0\right) \right\vert , \ \ \ \ \ \ \ 
\label{3.5}
\end{equation}
\end{description}

uniformly with respect to $(t,\omega )$ for each $x,y\in BC\left( \left[
-\tau ,0\right] ;\mathbb{R}\right) $.

Note that, since $\left\vert Q\left( 0,\eta \right) \right\vert \leq
k_{0}\left\vert \eta \left( 0\right) \right\vert +\left\vert Q\left(
0,0\right) \right\vert ,$ then $Q\left( 0,\eta \right) \in L_{G}^{2}\left(
\Omega _{T}\right) $ for all $\eta \in BC\left( \left[ -\tau ,0\right] ;%
\mathbb{R}\right) .$

\begin{remark}
Indeed, the functions $Q,$ $b,$ $\gamma,$ and $\sigma $ defined by%
\begin{eqnarray*}
b\left( t,x,u\right) &:&=c\left( t,x\left( 0\right) ,u\right) , \\
\gamma \left( t,x,u\right) &:&=\alpha \left( t,x\left( 0\right) ,u\right),
\\
\sigma \left( t,x\right) &:&=\beta \left( t,x\left( 0\right) \right) ,
\end{eqnarray*}
\end{remark}

\bigskip \textit{and}%
\begin{equation*}
Q\left( t,x\right) :=\lambda \left( t,x\left( 0\right) \right) ,
\end{equation*}%
\textit{such that, the functions }$c,\alpha $ and $\beta $ are $K_{1}$-
\textit{Lipschitz, and }$\lambda $\textit{\ is }$k_{0}$-\textit{Lipschitz }%
uniformly with respect to $(t,\omega )$ for each $x\in BC\left( \left[
-\tau ,0\right] ;\mathbb{R}\right) $\textit{, satisfies the assumptions }$%
\left( A_{1}\right) $\textit{\ and }$\left( A_{2}\right) .$

\begin{theorem}
Let the assumptions $\left( A_{1}\right) $ and $\left( A_{2}\right) $ are
satisfied. Then, for each $u\left( t\right) \in \mathbb{A},$ the integral
equation $(\ref{3.1})$ has an unique solution \ $X^{u}\in \widetilde{%
\mathcal{H}}_{T}$.
\end{theorem}

\begin{proof}
Let the mapping $\Theta :\widetilde{\mathcal{H}}_{T}\rightarrow \widetilde{%
\mathcal{H}}_{T}$ defined by: for each $t\in \left[ 0,T\right] ,$ 
\begin{eqnarray}
\Theta \left( X\right) \left( t\right) &=&\eta \left( 0\right) +Q\left(
t,X_{t}\right) -Q\left( 0,\eta \right) +\int_{0}^{t}b\left( s,X_{s},u\left(
s\right) \right) ds\   \notag \\
&&+\int_{0}^{t}\gamma \left( s,X_{s},u\left( s\right) \right) d\left\langle
B\right\rangle _{s}\ +\int_{0}^{t}\sigma \left( s,X_{s}\right) dB_{s}.
\label{3.6}
\end{eqnarray}

We have for all $X,\overline{X}\in \widetilde{\mathcal{H}}_{T}$%
\begin{eqnarray}
&&\left\vert \Theta \left( X\right) \left( t\right) -\Theta \left( \overline{%
X}\right) \left( t\right) \right\vert  \notag \\
&\leq &\left\vert Q\left( t,X_{t}\right) -Q\left( t,\overline{X}_{t}\right)
\right\vert +\left\vert \int_{0}^{t}\left[ b\left( s,X_{s},u\left( s\right)
\right) -b\left( s,\overline{X}_{s},u\left( s\right) \right) \right]
ds\right\vert \   \notag \\
&&+\left\vert \int_{0}^{t}\left[ \gamma \left( s,X_{s},u\left( s\right)
\right) -\gamma \left( s,\overline{X}_{s},u\left( s\right) \right) \right]
d\left\langle B\right\rangle _{s}\right\vert \ \   \notag \\
&&+\left\vert \int_{0}^{t}\left[ \sigma \left( s,X_{s}\right) -\sigma \left(
s,\overline{X}_{s}\right) \right] dB_{s}\right\vert .  \label{3.7}
\end{eqnarray}

Taking $G$-expectation on both sides, and using the following inequality%
\begin{equation}
\left( \sum_{i=1}^{k}d_{i}\right) ^{2}\leq 2^{k-1}\sum_{i=1}^{k}d_{i}^{2},\ 
\text{for each }d_{1}...d_{k}\ >0  \label{3.8}
\end{equation}

we have$\ \ \ \ $%
\begin{eqnarray}
&&\widehat{E}\left[ \left\vert \Theta \left( X\right) \left( t\right)
-\Theta \left( \overline{X}\right) \left( t\right) \right\vert ^{2}\right] 
\notag \\
&\leq &8\widehat{E}\left[ \left\vert Q\left( t,X_{t}\right) -Q\left( t,%
\overline{X}_{t}\right) \right\vert ^{2}\right] +8\widehat{E}\left[
\left\vert \int_{0}^{t}\left[ b\left( s,X_{s},u\left( s\right) \right)
-b\left( s,\overline{X}_{s},u\left( s\right) \right) \right] ds\right\vert
^{2}\ \right]  \notag \\
&&+8\widehat{E}\left[ \left\vert \int_{0}^{t}\left[ \gamma \left(
s,X_{s},u\left( s\right) \right) -\gamma \left( s,\overline{X}_{s},u\left(
s\right) \right) \right] d\left\langle B\right\rangle _{s}\right\vert ^{2}%
\right] \   \notag \\
&&+8\widehat{E}\left[ \left\vert \int_{0}^{t}\left[ \sigma \left(
s,X_{s}\right) -\sigma \left( s,\overline{X}_{s}\right) \right]
dB_{s}\right\vert ^{2}\right]  \notag \\
&:&=8\sum_{i=1}^{4}U_{i}.  \label{3.9}
\end{eqnarray}

Now, we have by assumption $\left( A_{2}\right) $%
\begin{equation}
U_{1}\leq k_{0}^{2}\widehat{E}\left[ \left\vert X\left( t\right) -\overline{X%
}\left( t\right) \right\vert ^{2}\ \right]  \label{3.10}
\end{equation}

By applying H\"{o}lder inequality and $\left( A_{1}\right) $, we have%
\begin{eqnarray}
U_{2} &\leq &T\int_{0}^{T}\widehat{E}\left[ \left\vert \left[ b\left(
s,X_{s},u\left( s\right) \right) -b\left( s,\overline{X}_{s},u\left(
s\right) \right) \right] \right\vert ^{2}\ \ \right] ds\   \notag \\
&\leq &TK_{1}^{2}\int_{0}^{T}\widehat{E}\left[ \left\vert X\left( s\right) -%
\overline{X}\left( s\right) \right\vert ^{2}\ \right] ds.  \label{3.11}
\end{eqnarray}

Similarly, by using the $G$-BDG inequalities, we obtain%
\begin{eqnarray*}
U_{3}+U_{4} &\leq &T\bar{\sigma}^{2}\int_{0}^{T}\widehat{E}\left[ \left\vert
\gamma \left( s,X_{s},u\left( s\right) \right) -\gamma \left( s,\overline{X}%
_{s},u\left( s\right) \right) \right\vert ^{2}\ \right] ds \\
&&+C_{2}\int_{0}^{T}\widehat{E}\left[ \left\vert \sigma \left(
s,X_{s}\right) -\sigma \left( s,\overline{X}_{s}\right) \right\vert ^{2}\ %
\right] ds,
\end{eqnarray*}%
\begin{equation}
=K_{1}^{2}\left[ T\bar{\sigma}^{2}+C_{2}\right] \int_{0}^{T}\widehat{E}\left[
\left\vert X\left( s\right) -\overline{X}\left( s\right) \right\vert ^{2}\ %
\right] ds.  \label{3.12}
\end{equation}

Combining $(\ref{3.10})$,$(\ref{3.11})$, and $(\ref{3.12})$, we get
\begin{eqnarray}
\widehat{E}\left[ \left\vert \Theta \left( X\right) \left( t\right) -\Theta
\left( \overline{X}\right) \left( t\right) \right\vert ^{2}\right] &\leq
&8k_{0}^{2}\widehat{E}\left[ \left\vert X\left( t\right) -\overline{X}\left(
t\right) \right\vert ^{2}\ \right]  \notag \\
&&+C\int_{0}^{T}\widehat{E}\left[ \left\vert X\left( s\right) -\overline{X}%
\left( s\right) \right\vert ^{2}\ \right] ds.  \label{3.13}
\end{eqnarray}

where $C=8K_{1}^{2}\left( T+T\bar{\sigma}^{2}+C_{2}\right) $.

Multiplying by $\exp \left( -2Ct\right) $ both sides of inequality $(\ref{3.13})$ and integrating on $\left[ 0,T\right] ,$ we obtain\ \ \ \ \ \ \
\ \ \ \ \ \ \ \ \ 
\begin{eqnarray}
N_{C}^{2}\left[ \Theta \left( X\right) -\Theta \left( \overline{X}\right) %
\right] &\leq &8k_{0}^{2}N_{C}^{2}\left[ X-\overline{X}\right]  \notag \\
&&+C\int_{0}^{T}\exp \left( -2Ct\right) \left( \int_{0}^{t}\widehat{E}\left[
\left\vert X\left( s\right) -\overline{X}\left( s\right) \right\vert ^{2}\ %
\right] ds\right) dt  \notag \\
&\leq &8k_{0}^{2}N_{C}^{2}\left[ X-\overline{X}\right]  \notag \\
&&+C\int_{0}^{T}\left( \widehat{E}\left[ \left\vert X\left( s\right) -%
\overline{X}\left( s\right) \right\vert ^{2}\ \right] \int_{s}^{T}\exp
\left( -2Ct\right) dt\right) ds  \notag \\
&\leq &8k_{0}^{2}N_{C}^{2}\left[ X-\overline{X}\right]  \notag \\
&&+\int_{0}^{T}\widehat{E}\left[ \left\vert X\left( s\right) -\overline{X}%
\left( s\right) \right\vert ^{2}\ \right] \left( \frac{e^{-2Cs}-e^{-2CT}}{2}%
\right) ds  \notag \\
&\leq &8k_{0}^{2}N_{C}^{2}\left[ X-\overline{X}\right] +\frac{1}{2}N_{C}^{2}%
\left[ X-\overline{X}\right].  \label{3.14}
\end{eqnarray}

Thus, we obtain the following estimation%
\begin{equation*}
N_{C}\left[ \Theta \left( X\right) -\Theta \left( \overline{X}\right) \right]
\leq \sqrt{8k_{0}^{2}+\frac{1}{2}}N_{C}\left[ X-\overline{X}\right] .
\end{equation*}

We have, by using H\"{o}lder inequality, 
\begin{eqnarray*}
N_{0}^{2}\left( \overset{.}{\underset{0}{\int }}b\left( s,0,u\left( s\right)
\right) ds\right) &=&\overset{T}{\underset{0}{\int }}\widehat{E}\left[
\left\vert \overset{t}{\underset{0}{\int }}b\left( s,0,u\left( s\right)
\right) ds\right\vert ^{2}\right] dt \\
&\leq &T\overset{T}{\underset{0}{\int }}\overset{t}{\underset{0}{\int }%
\widehat{E}}\left[ \left\vert b\left( s,0,u\left( s\right) \right)
\right\vert ^{2}ds\right] dt \\
&\leq &T^{2}N_{0}^{2}\left( b\left( .,0,u\left( .\right) \right) \right).
\end{eqnarray*}

Similarly, it easy to check, by $G$-BDG inequalities, that%
\begin{equation*}
N_{0}^{2}\left( \overset{.}{\underset{0}{\int }}\gamma \left( s,0,u\left(
s\right) \right) d\left\langle B\right\rangle _{s}\right) \leq \overline{%
\sigma }^{2}T^{2}N_{0}^{2}\left( \gamma \left( .,0,u\left( .\right) \right)
\right)
\end{equation*}

and%
\begin{equation*}
N_{0}^{2}\left( \overset{.}{\underset{0}{\int }}\sigma \left( s,0\right)
dB_{s}\right) \leq C_{2}TN_{0}^{2}\left( \sigma \left( .,0\right) \right).
\end{equation*}

Now observe that,%
\begin{eqnarray*}
\Theta \left( 0\right) \left( t\right) &=&\eta \left( 0\right) +Q\left(
t,0\right) -Q\left( 0,\eta \right) +\overset{t}{\underset{0}{\int }}b\left(
s,0,u\left( s\right) \right) ds \\
&&+\overset{t}{\underset{0}{\int }}\gamma \left( s,0,u\left( s\right)
\right) d\left\langle B\right\rangle _{s}+\overset{t}{\underset{0}{\int }}%
\sigma \left( s,0\right) dB_{s}.
\end{eqnarray*}

It follows that%
\begin{eqnarray*}
N_{0}\left( \Theta \left( 0\right) \right) &\leq &\sqrt{T}\left( \left\Vert
Q\left( 0,\eta \right) \right\Vert _{L_{G}^{2}\left( \Omega _{T}\right)
}+\left\vert \eta \left( 0\right) \right\vert \right) +N_{0}\left( Q\left(
.,0\right) \right) +TN_{0}\left( b\left( .,0,u\left( .\right) \right) \right)
\\
&&+\overline{\sigma }TN_{0}\left( \gamma \left( .,0,u\left( .\right) \right)
\right) +\sqrt{C_{2}T}N_{0}\left( \sigma \left( .,0\right) \right) ,
\end{eqnarray*}

then the process $\Theta \left( 0\right) \in \widetilde{\mathcal{H}}_{T}$,
so that if $X\in \widetilde{\mathcal{H}}_{T}$ then 
\begin{equation*}
N_{C}\left( \Theta \left( X\right) \right) \leq N_{C}\left( \Theta \left(
X\right) -\Theta \left( 0\right) \right) +N_{C}\left( \Theta \left( 0\right)
\right) \leq N_{0}\left( X\right) +N_{0}\left( \Theta \left( 0\right)
\right) <\infty .
\end{equation*}

This means that $\Theta \left( X\right) \in $ $\widetilde{\mathcal{H}}_{T}$,
which implies that $\Theta $ is well defined.

Finally, taking into account the fact that $\sqrt{8k_{0}^{2}+\frac{1}{2}}<1$
and assumption $(A_{2})$, we deduce that $\ \Theta \left( X\right) $ is a
contraction on $\widetilde{\mathcal{H}}_{T}$, then the fixed point $X^{u}\in 
\widetilde{\mathcal{H}}_{T}$ is the unique solution of $(\ref{3.2})$.
The proof is completed.
\end{proof}

\subsection{Relaxed control problem}

In this section, we consider a relaxed control problem $(\ref{3.2})$.
Let $X^{\mu }$ denotes the solution of equation $(\ref{3.2})$
associated with the relaxed control. We establish the existence of a
minimizer of the cost corresponding to $\mu$.

\begin{center}
\begin{equation*}
J(\mu )=\widehat{E}\left[ \int_{0}^{T}\int_{A}\mathcal{L}(t,X_{t}^{\mu },\xi
)\mu _{t}(d\xi )\,ds+\Psi (X_{T}^{\mu })\right] ,
\end{equation*}
\end{center}

the functions, 
\begin{eqnarray*}
\mathcal{L} &:&\left[ 0,T\right] \times BC\left( \left[ -\tau ,0\right] ;%
\mathbb{R}\right) \times \mathbb{A}\rightarrow \mathbb{R}, \\
\Psi &:&BC\left( \left[ -\tau ,0\right] ;\mathbb{R}\right) \mathcal{%
\longrightarrow }\mathbb{R},
\end{eqnarray*}%
satisfy the following assumption:

\begin{description}
\item[$\left( A_{3}\right) $] $\mathcal{L}$, $\Psi $ are bounded and for
each $t\in[0,T]$ and $x\in BC\left( \left[ -\tau ,0\right] ;\mathbb{R}%
\right) $ the functions $\mathcal{L}(s,x,\cdot)$, $\Psi(s,x,\cdot) $ are
continuous. Additionally, we suppose that:
\end{description}

\begin{equation*}
\left\vert \mathcal{L}\left( t,x,u\right) -\mathcal{L}\left(
t,y,u\right)\right\vert +\left\vert \Psi \left( x\right) -\Psi \left(
y\right)\right\vert \leq K_{1}\left\vert x\left( 0\right) -y\left(
0\right)\right\vert .
\end{equation*}

We recall that in the strict control problem

\begin{equation}
J(u)=\widehat{E}\left[ \int_{0}^{T}\mathcal{L}(t,X_{t}^{u},u\left( t\right)
)\,dt+\Psi (X_{T}^{u})\right]  \label{3.15}
\end{equation}%
\ \ \ \ \ 

over the set $\mathcal{U}$,

\begin{center}
\begin{eqnarray}
X^{u}\left( t\right) &=&\eta \left( 0\right) +Q\left( t,X_{t}^{u}\right)
-Q\left( 0,\eta \right) +\int_{0}^{t}b(s,X_{s}^{u},u\left( s\right) )ds\ \  
\notag \\
&&+\int_{0}^{t}\gamma (s,X_{s}^{u},u\left( s\right) ))d\left\langle
B\right\rangle _{s}\ +\int_{0}^{t}\sigma \left( s,X_{s}^{u}\right) dB_{s}
\label{3.16}
\end{eqnarray}%
\ \ \ \ \ \ \ 
\end{center}

then, we have

\begin{center}
\begin{eqnarray}
X^{\mu }\left( t\right) &=&\eta \left( 0\right) +Q\left( t,X_{t}^{\mu
}\right) -Q\left( 0,\eta \right) +\int_{0}^{t}\int_{\mathbb{A}%
}b(s,X_{s}^{\mu },\xi )ds\ \   \notag \\
&&+\int_{0}^{t}\int_{\mathbb{A}}\gamma (s,X_{s}^{\mu },\xi ))\mu _{s}(d\xi
)d\left\langle B\right\rangle _{s}\ +\int_{0}^{t}\sigma \left( s,X_{s}^{\mu
}\right) dB_{s}.  \label{3.17}
\end{eqnarray}%
$\ $
\end{center}

We suppose as well that the coefficients of the $G$-NSFDE verify the
following condition
\begin{description}
\item[$\left( A_{4}\right) $] The coefficients $b,\gamma ,\sigma $ are bounded and for every fixed $t\in \lbrack 0,T]$ and $x\in BC\left( \left[ -\tau ,0\right] ;\mathbb{R}\right) $ the functions $b(t,x,\cdot ),\gamma (t,x,\cdot )$ are continuous $q.s$.
\end{description}
\subsection{Approximation and existence of relaxed optimal control}

By introducing the relaxed control problem, the next lemma, which extends
the celebrated Chattering Lemma, states that each relaxed control in $%
\mathcal{R}$ can be approximated by strict controls$.$

\begin{definition}
(\textbf{stable convergence}) Let $\mu ^{n},\mu \in \mathcal{R},n\in \mathbb{%
N}^{\ast }$. We say that, we have a stable convergence, if for any
continuous function $f:\left[ 0,T\right] \times \mathbb{A}\rightarrow 
\mathbb{R},$ we have%
\begin{equation}
\underset{n\rightarrow \infty }{\lim }\int_{\left[ 0,T\right] \times \mathbb{%
A}}f\left( t,\xi \right) \mu ^{n}\left( dt,d\xi \right) =\int_{\left[ 0,T%
\right] \times \mathbb{A}}f\left( t,\xi \right) \mu \left( dt,d\xi \right)
\label{3.18}
\end{equation}
\end{definition}

\begin{lemma}\label{gch36}
(\cite{redjil2018}) (\textbf{(G-Chattering Lemma)}) Let $(\mathbb{A},d)$ be
a separable compact metric space.\ Let $(\mu _{t})_{t\geq 0}$ be an $\mathbb{%
F}^{\mathcal{P}}$-progressively measurable process taking values in $%
\mathcal{P}(\mathbb{A})$. Then there exists a sequence $(u^{n}(t))_{n\geq 0}$
of $\mathbb{F}^{\mathcal{P}}$-progressively measurable processes taking
values in $\mathbb{A},$ such that the sequence of random measures $\delta {%
_{u^{n}{(t)}}}(d\xi )dt$ converges in the sense of stable convergence (thus
weakly) to $\mu _{t}(d\xi )dt\,\,$ $q.s.$
\end{lemma}

Taking use of the fact that under $\mathbb{P}\in $ $\mathcal{P}$, $B$ is a
continuous martingale with a quadratic variation process $\left\langle
B\right\rangle $ such that $c_{t}$ $:=\frac{d\left\langle B\right\rangle _{t}%
}{dt}$ is bounded. \ Let $X^{\mu }$ and $X^{n}$ the corresponding solutions
satisfy the following integral equations type of $G$-NSFDEs:%
\begin{eqnarray}
X^{\mu }(t) &=&\eta \left( 0\right) +Q\left( t,X_{t}^{\mu }\right) -Q\left(
0,\eta \right)  \label{3.19} \\
&&+\int_{0}^{t}\int_{\mathbb{A}}(b(s,X_{s}^{\mu },\xi )+c_{s}\gamma
(s,X_{s}^{\mu },\xi ))\mu _{s}(d\xi )ds+\int_{0}^{t}\sigma (s,X_{s}^{\mu
})dB_{s}  \notag
\end{eqnarray}

and%
\begin{eqnarray}
X^{n}(t) &=&\eta \left( 0\right) +Q\left( t,X_{t}^{n}\right) -Q\left( 0,\eta
\right)  \label{3.20} \\
&&+\int_{0}^{t}(b(s,X_{s}^{n},\xi )+c_{s}\gamma (s,X_{s}^{n},\xi ))\delta
_{u^{n}\left( s\right) }\left( d\xi \right) ds+\int_{0}^{t}\sigma
(s,X_{s}^{n})dB_{s}  \notag
\end{eqnarray}%
$\ \ \ \ \ \ \ \ \ \ \ \ \ \ \ $

with random initial data%
\begin{equation}
X_{0}^{\mu }=X_{0}^{n}=\eta \in BC\left( \left[ -\tau ,0\right] ;\mathbb{R}%
\right).  \label{3.21}
\end{equation}

\begin{lemma}\label{lsr37}
(\textbf{stability results})

Let $\mu $ be a relaxed control,\ and let $\left( u^{n}\right) $ be a
sequence defined as in \textbf{(G-Chattering Lemma)}. Then we have
\begin{description}
\item[$\left( i\right) $] For every$\ \mathbb{P}\in \mathcal{P},$ \textit{it holds that}
\begin{equation}
\underset{n\rightarrow \infty }{\lim }E^{\mathbb{P}}\left[ \underset{0\leq
t\leq T}{\sup }\left\vert X^{n}(t)-X^{\mu }(t)\right\vert ^{2}\right] =0
\label{3.22}
\end{equation}

and%
\begin{equation}
\underset{n\rightarrow \infty }{\lim }\widehat{E}\left[ \underset{0\leq
t\leq T}{\sup }\left\vert X^{n}(t)-X^{\mu }(t)\right\vert ^{2}\right] =0.
\label{3.23}
\end{equation}%
$\ \ \ $

\item[$\left( ii\right) $]  Let $J(u^{n})$ and $J(\mu )$ \ be the corresponding cost functionals to $u^{n}$ and $\mu $ respectively. Then, there exists a subsequence $\left( u^{n_{k}}\right) $ of $\left( u^{n}\right) $ such that for every$\ \mathbb{P}\in \mathcal{P}$%
\begin{equation}
\ \underset{k\rightarrow \infty }{\lim }J^{\mathbb{P}}(u^{n_{k}})=J^{\mathbb{%
P}}(\mu )  \label{3.24}
\end{equation}%
$\ \ \ \ \ \ \ \ \ \ \ $ $\ \ \ \ $ \ \ \ 

and%
\begin{equation}
\underset{k\rightarrow \infty }{\lim }J(u^{n_{k}})=J(\mu ).  \label{3.25}
\end{equation}

\textit{Moreover},%
\begin{equation}
\inf_{u\in \mathcal{U}}J^{\mathbb{P}}(u)=\inf_{\mu \in \mathcal{R}}J^{%
\mathbb{P}}(\mu )  \label{3.26}
\end{equation}%
$\ \ \ $

\textit{and there exists a relaxed control } $\widehat{\mu }_{\mathbb{P}}\in 
\mathcal{R}$ such that%
\begin{equation}
J^{\mathbb{P}}(\widehat{\mu }_{\mathbb{P}})=\inf_{\mu \in \mathcal{R}}J^{%
\mathbb{P}}(\mu ).  \label{3.27}
\end{equation}

\end{description}

\end{lemma}

\begin{proof}

\begin{description}
\item[$\left( i\right) $] The proof of this result is inspired by \cite
{redjil2018}. Subtracting $(\ref{3.19})$ from $(\ref{3.20})$ term by term, we have%
\begin{eqnarray}
X^{n}(t)-X^{\mu }(t) &=&\left[ Q\left( t,X_{t}^{n}\right) -Q\left(
t,X_{t}^{\mu }\right) \right]  \notag \\
&&+\int_{0}^{t}\int_{\mathbb{A}}(b(s,X_{s}^{n},\xi )+c_{s}\gamma
(s,X_{s}^{n},\xi ))\delta _{u^{n}\left( s\right) }\left( d\xi \right) ds 
\notag \\
&&-\int_{0}^{t}\int_{\mathbb{A}}(b(s,X_{s}^{\mu },\xi )+c_{s}\gamma
(s,X_{s}^{\mu },\xi ))\mu _{s}(d\xi )ds  \notag \\
&&+\int_{0}^{t}\left[ \sigma (s,X_{s}^{n})-\sigma (s,X_{s}^{\mu })\right]
dB_{s}  \notag \\
&=&\left[ Q\left( t,X_{t}^{n}\right) -Q\left( t,X_{t}^{\mu }\right) \right] +%
\mathcal{I}_{n}\left( s\right).  \label{3.28}
\end{eqnarray}

Taking $G$-expectation on both sides and using the assumptions $\left(
A_{1}\right) $ and $\left( A_{2}\right) ,$ it follows that%
\begin{center}
\begin{eqnarray}
\widehat{E}\left[ \underset{0\leq t\leq T}{\sup }\left\vert X^{n}\left(
t\right) -X^{\mu }\left( t\right) \right\vert ^{2}\right] &\leq & 2k_{0}^{2}%
\widehat{E}\left[ \underset{0\leq t\leq T}{\sup }\left\vert X^{n}\left(
t\right) -X^{\mu }\left( t\right) \right\vert ^{2}\right] \notag \\
&+& 2\widehat{E}\left[\underset{0\leq t\leq T}{\sup }\left\vert \mathcal{I}_{n}\left( s\right)\right\vert ^{2}\right]  \notag
\end{eqnarray}
\end{center}

then%
\begin{equation}
\ \widehat{E}\left[ \underset{0\leq t\leq T}{\sup }\left\vert X^{n}\left(
t\right) -X^{\mu }\left( t\right) \right\vert ^{2}\right] \ \ \leq \frac{2}{%
\left( 1-2k_{0}^{2}\right) }\widehat{E}\left[ \underset{0\leq t\leq T}{\sup }%
\left\vert \mathcal{I}_{n}\left( s\right) \right\vert ^{2}\right].  \label{3.29}
\end{equation}

We have%
\begin{eqnarray}
&&\widehat{E}\left[ \underset{0\leq t\leq T}{\sup }\left\vert \mathcal{I}%
_{n}\left( s\right) \right\vert ^{2}\right]  \notag \\
&\leq &\widehat{E}\left( \underset{0\leq t\leq T}{\sup }\left\vert
\int_{0}^{t}\int_{\mathbb{A}}(b(s,X_{s}^{n},\xi )+c_{s}\gamma
(s,X_{s}^{n},\xi ))\delta _{u^{n}\left( s\right) }\left( d\xi \right)
ds\right. \right.  \notag \\
&&\left. \left. -\int_{0}^{t}\int_{\mathbb{A}}(b(s,X_{s}^{\mu },\xi
)+c_{s}\gamma (s,X_{s}^{\mu },\xi ))\mu _{s}(d\xi )ds\right\vert ^{2}\right)
\notag \\
&&+\widehat{E}\left( \underset{0\leq t\leq T}{\sup }\left\vert \int_{0}^{t}%
\left[ \sigma (s,X_{s}^{n})-\sigma (s,X_{s}^{\mu })\right] dB_{s}\
\right\vert ^{2}\right) \ \ \   \notag \\
&\leq &\widehat{E}\left( \underset{0\leq t\leq T}{\sup }\left\vert
\int_{0}^{t}\int_{\mathbb{A}}(b(s,X_{s}^{n},\xi )+c_{s}\gamma
(s,X_{s}^{n},\xi ))\delta _{u^{n}\left( s\right) }\left( d\xi \right)
ds\right. \right.  \notag \\
&&\left. \left. -\int_{0}^{t}\int_{\mathbb{A}}(b(s,X_{s},\xi )+c_{s}\gamma
(s,X_{s},\xi ))\delta _{u^{n}\left( s\right) }(d\xi )ds\right\vert
^{2}\right)  \notag \\
&&+\widehat{E}\left( \underset{0\leq t\leq T}{\sup }\left\vert
\int_{0}^{t}\int_{\mathbb{A}}(b(s,X_{s},\xi )+c_{s}\gamma (s,X_{s},\xi
))\delta _{u^{n}\left( s\right) }\left( d\xi \right) ds\right. \right. 
\notag \\
&&\left. \left. -\int_{0}^{t}\int_{\mathbb{A}}(b(s,X_{s}^{\mu },\xi
)+c_{s}\gamma (s,X_{s}^{\mu },\xi ))\mu _{s}(d\xi )ds\right\vert ^{2}\right)
\notag \\
&&+\widehat{E}\left( \underset{0\leq t\leq T}{\sup }\left\vert \int_{0}^{t}%
\left[ \sigma (s,X_{s}^{n})-\sigma (s,X_{s}^{\mu })\right] dB_{s}\
\right\vert ^{2}\right).   \label{3.30}
\end{eqnarray}

\ \ \ \ Let $\varepsilon >0.$ Then, there exists $\mathbb{P}^{\varepsilon
}\in \mathcal{P}$ such that%
\begin{eqnarray}
&&\widehat{E}\left[ \underset{0\leq t\leq T}{\sup }\left\vert \mathcal{I}%
_{n}\left( s\right) \right\vert ^{2}\right] \notag \\
&&\leq E^{\mathbb{P}^{\varepsilon }}\left( \underset{0\leq t\leq T}{\sup }%
\left\vert \int_{0}^{t}\int_{\mathbb{A}}(b(s,X_{s}^{n},\xi )+c_{s}\gamma
(s,X_{s}^{n},\xi ))\delta _{u^{n}\left( s\right) }\left( d\xi \right)
ds\right. \right.  \notag \\
&&\left. \left. -\int_{0}^{t}\int_{\mathbb{A}}(b(s,X_{s},\xi )+c_{s}\gamma
(s,X_{s},\xi ))\delta _{u^{n}\left( s\right) }(d\xi )ds\right\vert
^{2}\right)  \notag \\
&&+E^{\mathbb{P}^{\varepsilon }}\left( \underset{0\leq t\leq T}{\sup }%
\left\vert \int_{0}^{t}\int_{\mathbb{A}}b(s,X_{s},\xi )\delta _{u^{n}\left(
s\right) }\left( d\xi \right) ds-\int_{0}^{t}\int_{\mathbb{A}}b(s,X_{s}^{\mu
},\xi )\mu _{s}(d\xi )ds\right\vert ^{2}\right)  \notag \\
&&+E^{\mathbb{P}^{\varepsilon }}\left( \underset{0\leq t\leq T}{\sup }%
\left\vert \int_{0}^{t}\int_{\mathbb{A}}c_{s}\gamma (s,X_{s},\xi )\delta
_{u^{n}\left( s\right) }\left( d\xi \right) ds \right. \right.   \notag \\
&&\left. \left.-\int_{0}^{t}\int_{\mathbb{A}%
}c_{s}\gamma (s,X_{s}^{\mu },\xi ))\mu _{s}(d\xi )ds\right\vert ^{2}\right)  \notag \\
&&+E^{\mathbb{P}^{\varepsilon }}\left( \underset{0\leq t\leq T}{\sup }%
\left\vert \int_{0}^{t}\left[ \sigma (s,X_{s}^{n})-\sigma (s,X_{s}^{\mu })%
\right] dB_{s}\ \right\vert ^{2}\right) +\varepsilon .  \label{3.31}
\end{eqnarray}

Then, we have%
\begin{eqnarray}
&&\widehat{E}\left[ \underset{0\leq t\leq T}{\sup }\left\vert \mathcal{I}%
_{n}\left( s\right) \right\vert ^{2}\right]  \notag \\
&\leq &16E^{\mathbb{P}^{\varepsilon }}\left( \underset{0\leq t\leq T}{\sup }%
\left\vert \int_{0}^{t}\int_{\mathbb{A}}(b(s,X_{s}^{n},\xi )+c_{s}\gamma
(s,X_{s}^{n},\xi ))\delta _{u^{n}\left( s\right) }\left( d\xi \right)
ds\right. \right.  \notag \\
&&\left. \left. -\int_{0}^{t}\int_{\mathbb{A}}(b(s,X_{s},\xi )+c_{s}\gamma
(s,X_{s},\xi ))\delta _{u^{n}\left( s\right) }(d\xi )ds\right\vert
^{2}\right)  \notag \\
&&+16E^{\mathbb{P}^{\varepsilon }}\left( \underset{0\leq t\leq T}{\sup }%
\left\vert \left[ \int_{0}^{t}\int_{\mathbb{A}}(b(s,X_{s},\xi )\delta
_{u^{n}\left( s\right) }\left( d\xi \right) ds\right. \right. \right.  \notag \\
&&\left. \left. \left. -\int_{0}^{t}\int_{\mathbb{A}%
}b(s,X_{s}^{\mu },\xi )\mu _{s}(d\xi )ds\right] \ \right\vert ^{2}\right) 
\notag \\
&&+16E^{\mathbb{P}^{\varepsilon }}\left( \underset{0\leq t\leq T}{\sup }%
\left\vert \left[ \int_{0}^{t}\int_{\mathbb{A}}c_{s}\gamma (s,X_{s},\xi
)\delta _{u^{n}\left( s\right) }\left( d\xi \right) ds\right. \right. \right.  \notag \\
&&\left. \left. \left.-\int_{0}^{t}\int_{\mathbb{A}}c_{s}\gamma (s,X_{s}^{\mu },\xi ))\mu _{s}(d\xi )ds\right] \right\vert ^{2}\right)  \notag \\
&&+16E^{\mathbb{P}^{\varepsilon }}\left( \underset{0\leq t\leq T}{\sup }%
\left\vert \int_{0}^{t}\left[ \sigma (s,X_{s}^{n})-\sigma (s,X_{s}^{\mu })%
\right] dB_{s}\ \right\vert ^{2}\right) +16\varepsilon ^{2}  \notag \\
&=&16\left\{ \left( \mathcal{I}_{\left( n,1\right) }+\mathcal{I}_{\left(
n,2\right) }+\mathcal{I}_{\left( n,3\right) }+\mathcal{I}_{\left( n,4\right)
}\right) +\varepsilon ^{2}\ \right\}. \label{3.32}
\end{eqnarray}

Since \ $b$, $\gamma $ are bounded and continuous in the control variable $%
\xi $, then, by using the dominated convergence theorem, and the stable
convergence of $\delta _{u^{n}\left( t\right) }\left( d\xi \right) dt$ to $%
\mu _{t}(d\xi )dt$, we have%
\begin{equation}
\underset{n\rightarrow \infty }{\lim }\mathcal{I}_{\left( n,2\right) }=%
\underset{n\rightarrow \infty }{\lim }\mathcal{I}_{\left( n,3\right) }=0.
\label{3.33}
\end{equation}
Similarly, we use the assumption $\left( A_{1}\right) ,$ then%

\begin{eqnarray}
\underset{n\rightarrow \infty }{\lim }\left( \mathcal{I}_{\left( n,1\right)
}+\mathcal{I}_{\left( n,4\right) }\right) \leq K_{1}^{2}\underset{n\rightarrow \infty }{\lim }\left[ E^{\mathbb{P}^{\varepsilon }}\left(\int_{0}^{T}\left\vert X^{n}\left( s\right) -X^{\mu }\left( s\right)\right\vert ^{2}\right) dt+\varepsilon ^{2}\right].\label{3.34}
\end{eqnarray}

It follows, by using dominated convergence theorem, that%
\begin{eqnarray}
\underset{n\rightarrow \infty }{\lim }E^{\mathbb{P}^{\varepsilon }}\left[ 
\underset{0}{\overset{T}{\int }}\left\vert X^{n}\left( s\right) -X^{\mu
}\left( s\right) \right\vert ^{2}ds\right] \leq \underset{0}{\overset{T}{\int }}\underset{n\rightarrow \infty }{\lim }E^{\mathbb{P}^{\varepsilon }}\left[ \left\vert X^{n}\left( s\right) -X^{\mu}\left( s\right) \right\vert ^{2}\right] ds  \notag \\
\leq\underset{0}{\overset{T}{\int }}\underset{n\rightarrow \infty }{\lim }
\widehat{E}\left[ \underset{0\leq \upsilon \leq s}{\sup }\left\vert
X^{n}\left( \nu \right) -X^{\mu }\left( \nu \right) \right\vert ^{2}\right]
ds. \label{3.35}
\end{eqnarray}

Taking $Z\left( \delta \right) =\underset{n\rightarrow \infty }{\lim }%
\widehat{E}\left[ \underset{0\leq t\leq \delta }{\sup }\left\vert
X^{n}\left( t\right) -X^{\mu }\left( t\right) \right\vert ^{2}\right] $, for
each $\delta >0$, then we deduce from the formulas $(\ref{3.34})$ and $(\ref{3.35})$, that%
\begin{equation}
Z\left( T\right) \leq \frac{32K_{1}^{2}}{1-2k_{0}^{2}}\left( \underset{0}{%
\overset{T}{\int }}Z\left( s\right) ds+\varepsilon ^{2}\right)  \label{3.36}
\end{equation}

using Gronwall's lemma, we conclude that%
\begin{equation}
\underset{n\rightarrow \infty }{\lim }\widehat{E}\left[ \underset{0\leq
t\leq T}{\sup }\left\vert X^{n}(t)-X^{\mu }(t)\right\vert ^{2}\right] =0.
\label{3.37}
\end{equation}
\item[$\left( ii\right) $] Property $\left( i\right) $ implies that there exists a subsequence $\left( X^{n_{k}}(t)\right) _{^{n_{k}}}$ that converges to $X^{\mu }(t)$ $q.s.$, and uniformly in $t$. We have, for all $\ \mathbb{P}\in \mathcal{P}$%
\begin{eqnarray}
\left\vert J^{\mathbb{P}}(u^{n_{k}})-J^{\mathbb{P}}(\mu )\right\vert &\leq
&E^{\mathbb{P}}\left[ \int_{0}^{T}\int_{\mathbb{A}}\left\vert \mathcal{L}%
(t,X_{t}^{n_{k}},\xi )-\mathcal{L}(t,X_{t}^{\mu },\xi )\right\vert \delta
_{u^{n_{k}}(t)}(d\xi )dt\right]  \notag \\
&&+E^{\mathbb{P}}\left[ \left\vert \int_{0}^{T}\int_{\mathbb{A}}\mathcal{L}%
(t,X_{t}^{\mu },\xi )\delta _{u^{n_{k}}(t)}(d\xi )dt\right. \right. \notag \\
&&\left. \left. -\int_{0}^{T}\int_{\mathbb{A}}\mathcal{L}(t,X_{t}^{\mu },\xi )\mu _{t}(d\xi )dt\right\vert \right]  \notag \\
&&+E^{\mathbb{P}}\left[ \left\vert \Psi (X_{T}^{n_{k}})-\Psi (X_{T}^{\mu
})\right\vert \right].  \label{3.38}
\end{eqnarray}

The first and third terms in the right-hand side converge to $0$ as a result
of the continuity and boundness assumptions on $\mathcal{L}$\ and $\Psi $
with respect to $X$\textit{\ }that. And, the second term on the right-hand
side tends to $0$, due to the continuity and the boundness of $\mathcal{L}$
in the variable $\xi $, and by the weak convergence of $\delta
_{u^{n_{k}}(t)}(d\xi )dt$ to $\mu _{t}(d\xi )dt,$ we use the dominated
convergence theorem to conclude.\newline
Using \textbf{Lemma} $\ref{lsr37}$ \textbf{(stability results)}, we obtain for all $\ \mathbb{P}\in \mathcal{P}$, 

\begin{equation}
\underset{k\rightarrow \infty }{\lim }J^{\mathbb{P}}(u^{n_{k}})=J^{\mathbb{P}%
}(\mu )  \label{3.39}
\end{equation}

then,%
\begin{equation}
\underset{k\rightarrow \infty }{\lim }J(u^{n_{k}})=J(\mu ),  \label{3.40}
\end{equation}

we have $J^{\mathbb{P}}(u)=J^{\mathbb{P}}(\delta _{u})$, This yields $%
\underset{u\in \mathcal{U}}{\inf }J^{\mathbb{P}}(u)\geq \underset{\mu \in R}{%
\inf }J^{\mathbb{P}}(\mu )$. Given an arbitrary $\mu \in \mathcal{R}$. From 
\textbf{Lemma} $\ref{gch36}$ \textbf{(G-Chattering Lemma)}, to obtain a sequence of strict controls $\left( u^{n_{k}}\right) \subset \mathcal{U}$ such that $\delta_{u^{n_{k}}(t)}(d\xi )dt$ converges weakly to $\mu _{t}(d\xi )dt$, we obtain 
\begin{equation}
J^{\mathbb{P}}(\mu )=\lim_{n\rightarrow \infty }J^{\mathbb{P}}(u^{n})\geq
\inf_{u\in \mathcal{U}}J^{\mathbb{P}}(u)  \label{3.41}
\end{equation}

since $\mu $ is arbitrary, we have: 
\begin{equation}
\inf_{\mu \in R}J^{\mathbb{P}}(\mu )\geq \inf_{u\in \mathcal{U}}J^{\mathbb{P}
}(u).  \label{3.42}
\end{equation}
\end{description}
\end{proof}

The main result is to give the following theorem. Note that this result
extends to $G$-NSFDEs with an uncontrolled diffusion coefficient. We show
that an optimal solution for the relaxed control problem exists, the proof
is based of the existence of optimal relaxed control for each $\mathbb{P\in }\mathcal{P}$ and a tightness argument.

\begin{theorem}
\label{relaxed} For every $u\in \mathcal{U}$ and $\mu \in \mathcal{R}$, we
have 
\begin{equation}
\inf_{u\in \mathcal{U}}J(u)=\inf_{\mu \in \mathcal{R}}J(\mu ).  \label{3.43}
\end{equation}

Moreover, there exists a relaxed control \textit{\ }$\widehat{\mu }\in 
\mathcal{R}$ such that%
\begin{equation}
J(\widehat{\mu })=\min_{\mu \in \mathcal{R}}J(\mu )  \label{3.44}
\end{equation}

recall that%
\begin{equation}
J(\mu )=\sup_{\mathbb{P}\in \mathcal{P}}J^{\mathbb{P}}(\mu )  \label{3.45}
\end{equation}

where for each $\mathbb{P}\in \mathcal{P}$, the relaxed cost functional is
given as follow%
\begin{equation}
J^{\mathbb{P}}(\mu )=E^{\mathbb{P}}\left[ \int_{0}^{T}\int_{\mathbb{A}}%
\mathcal{L}(t,X_{t}^{\mu },\xi )\mu _{t}(d\xi )\,dt+\Psi (X_{T}^{\mu })%
\right].  \label{3.46}
\end{equation}

Let $(\mu ^{n},$ $X^{\mu ^{n}})_{n\geq 0}$ be a minimizing sequence of $%
\underset{\mu \in \mathcal{R}}{\inf }J^{\mathbb{P}}(\mu )$ such that%
\begin{equation}
\underset{n\rightarrow \infty }{\lim }J^{\mathbb{P}}(\mu ^{n})=\inf_{\mu \in 
\mathcal{R}}J^{\mathbb{P}}(\mu )  \label{3.47}
\end{equation}

where $X^{\mu ^{n}}$ is the unique solution of $(\ref{3.2})$,
corresponding to the random variables $\mu ^{n}$ which belongs to the
compact set $M$. $\ \ \ \ \ \ \ \ \ $
\end{theorem}

The proof of the existence of an optimal relaxed control entails
demonstrating that the sequence of distributions of the processes $(\mu
^{n}, $ $X^{\mu^{n}})_{n\geq 0}$ is tight for a given topology on the state
space and then proving that we can extract a subsequence that converges in
law to a process $(\mu ,$ $X^{\mu })$, that satisfies $(\ref{3.17})$.
To achieve the proof, we show that under some regularity conditions of $%
\left( J^{\mathbb{P}}(\mu ^{n})\right) _{n}$ converges to $J^{\mathbb{P}}(%
\widehat{\mu })$ which is equal to $\underset{\mu \in \mathcal{R}}{\inf }J^{%
\mathbb{P}}(\mu )$ and then $(\widehat{\mu },$ $X^{\widehat{\mu }})$ is
optimal.

\begin{lemma}(\cite{bahlali2014}) 
The sequence of distributions of the relaxed controls $\left( \mu ^{n}\right) _{n\geq 0}$ is relatively compact in $M$.
\end{lemma}

\begin{proof}[Proof of Theorem \ref{relaxed}]
 The relaxed controls $\mu ^{n}$ are random variables in the compact set $M$. Then by Prohorov's theorem the associated family of distribution $\left( \mu ^{n}\right) _{n\geq 0}$ is tight on the space $M$, then it is relatively compact in $M$. Thus, there exists a subsequence $(\mu ^{n_{k}},X^{\mu ^{n_{k}}})_{k\geq 0}$ of $(\mu^{n},X^{\mu ^{n}})_{n\geq 0}$ that weakly converges to $(\widehat{\mu },X^{\widehat{\mu }})$ which solves $(\ref{3.17})$. Using Skorohod's embedding theorem, the continuity and boundness assumptions of the functions  $\mathcal{L}$ and $\Psi $, and Lebesgue Dominated Convergence Theorem, we finally obtain : 
\begin{equation*}
\inf_{\mu \in \mathcal{R}}J^{\mathbb{P}}(\mu )=\underset{k\rightarrow \infty 
}{\lim }J^{\mathbb{P}}(\mu ^{n_{k}})=J^{\mathbb{P}}(\widehat{\mu }).
\end{equation*}
Then, from \textbf{Lemma} $\ref{lsr37}$ \textbf{(stability results)}, for every $\mathbb{P}\in \mathcal{P}$ there exists a relaxed control $\widehat{\mu }\in \mathcal{R}$ such that%
\begin{equation*}
\widehat{\mu }_{\mathbb{P}}=\arg \min_{\mu \in \mathcal{R}}J^{\mathbb{P}%
}(\mu ).
\end{equation*}
Then, we conclude that%
\begin{equation*}
J(\widehat{\mu })=\min_{\mu \in \mathcal{R}}J(\mu ).
\end{equation*}
\end{proof}
\begin{remark}
\bigskip The relaxed model is a real extension of the strict model, as the
infimum of the two cost functions are equal, and the relaxed model has an
optimal solution, as shown by the prior results.
\end{remark}
\section{Numerical analysis: Euler-Maruyama method for G-Neutral SFDEs}\label{na4}

In this section we present a numerical analysis of a $G$-NSFDE. The idea is to use the Euler-Maruyama scheme to solve the G-NSFDE $(\ref{UCOGNSDE})$. Let $%
\tau >0,T>\tau ,N\in \mathbb{N},h=\frac{T+\tau }{N}$ and $t_{0}=-\tau
,t_{1}=-\tau +h,\cdots ,t_{N_{0}}=0,\cdots ,t_{N}=T$ be a discretization of
the interval $[-\tau ,T].$ Consider the following Euler-Maruyama scheme: 

\begin{equation}
\left\{ 
\begin{array}{l}
\mbox{Given an initial data }\eta :[-\tau ,0]\rightarrow \mathbb{R}^{n},%
\mbox{ and put }X(t)=\eta (t)\mbox{ for }t=t_{0},t_{1},\cdots t_{N_{0}} \\ 
\mbox{Now for }i=N_{0},N_{0}+1\cdots N \\ 
X\left( t_{i+1}\right) =X\left( t_{i}\right) +Q\left(
t_{i+1},X_{t_{i+1}}\right) -Q\left( t_{i},X_{t_{i}}\right) +b\left(
t_{i},X_{t_{i}}\right) h+\gamma \left( t_{i},X_{t_{i}}\right) \left\langle
B\right\rangle _{t_{i}}\  \\ 
+\sigma \left( t_{i},X_{t_{i}}\right) (B_{t_{i+1}}-B_{t_{i+1}})
\end{array}
\right. \ \   \label{Euler-M}
\end{equation}%
where, $X_{t_{i}}:=\{X_{t_{i}}(\lambda ):-\tau \leq \lambda \leq 0\},$ $%
X_{t_{i}}(\lambda ):=X(t_{i+k})+\frac{\lambda -t_{k}}{h}%
[X(t_{i+k+1})-X(t_{i+k})],k$ is such that $t_{k}\leq \lambda <t_{k+1}.$ In
order that our algorithm works we have to give a value for $X_{t_{0}-h},$ we
can set it equal $X_{t_{0}-h}=X_{t_{0}}=\eta (-\tau ).$

For the simulation of the increments of the $G$-Brownian motion and its
quadratic variation we follow the same method given by \cite{yang2016}
by simulating its corresponding $G$-PDE using finite difference.

In Figure \ref{density_different_sigmaMax} (resp. Figure %
\ref{distribution_different_sigmaMax}) we represent the simulation of the
density (resp. distribution) of the $G$-Normal BM for $\sigma _{min}=0.8$
and different values of $\sigma _{max}$.

Also, in Figure \ref{density_different_sigmaMin} (resp. Figure %
\ref{distribution_different_sigmaMin}) we represent the simulation of the
density (resp. distribution) of the $G$-Normal BM for $\sigma _{max}=1.3$
and different values of $\sigma _{min}.$

\begin{figure}[!ht]
\begin{center}
\textbf{Simulation of the G-Normal density and distribution for $%
\sigma_{min}=0.8$ and different $\sigma_{max}$}
\par
\medskip
\end{center}
\par
\begin{minipage}{0.43\textwidth}
		\centering
		\includegraphics[width=1.2\textwidth]{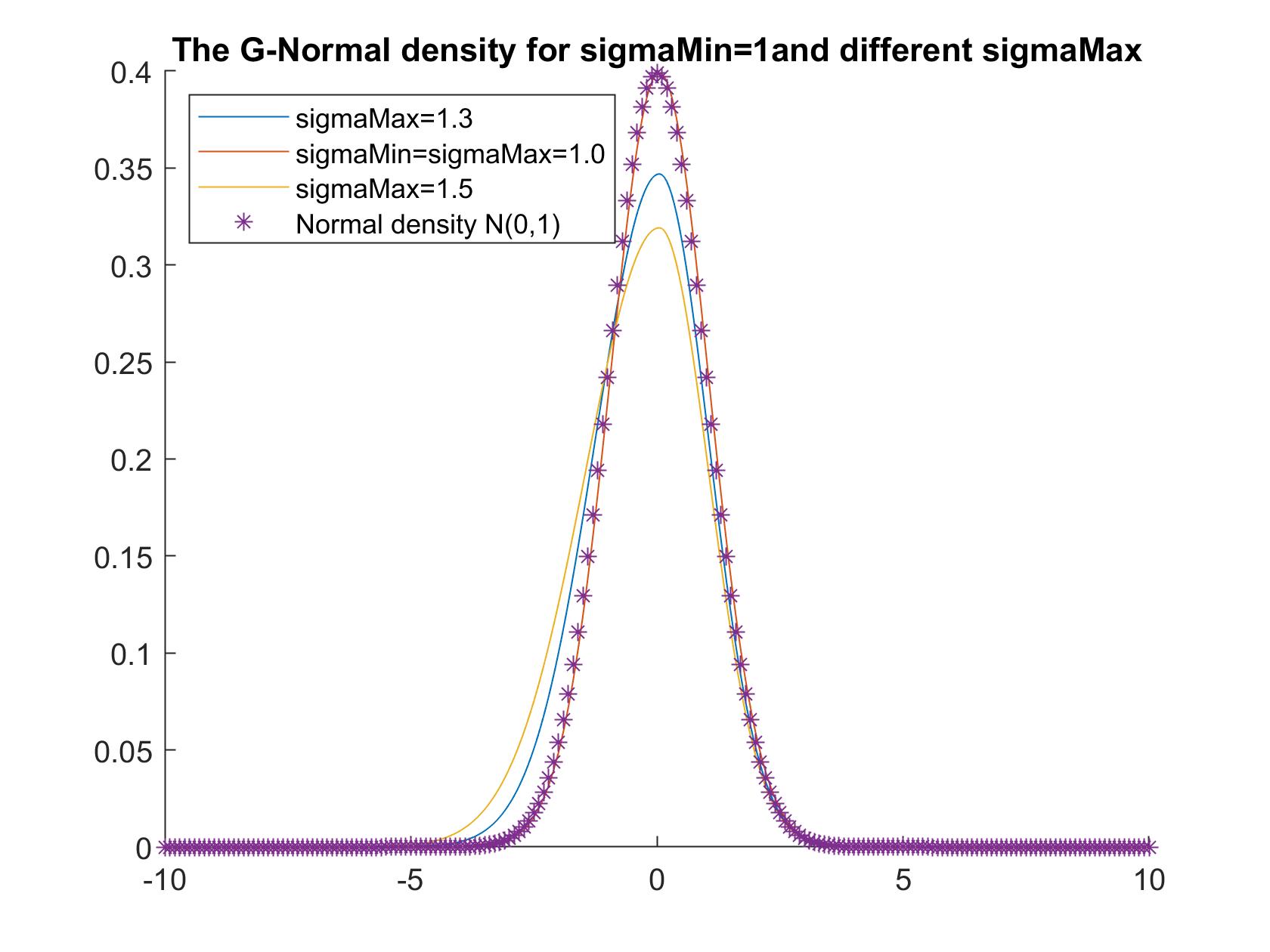}
		\caption{G-Normal density} \label{density_different_sigmaMax}
	\end{minipage}		
\hfill 
\begin{minipage}{0.43\textwidth}
		\centering
		\includegraphics[width=1.85\textwidth]{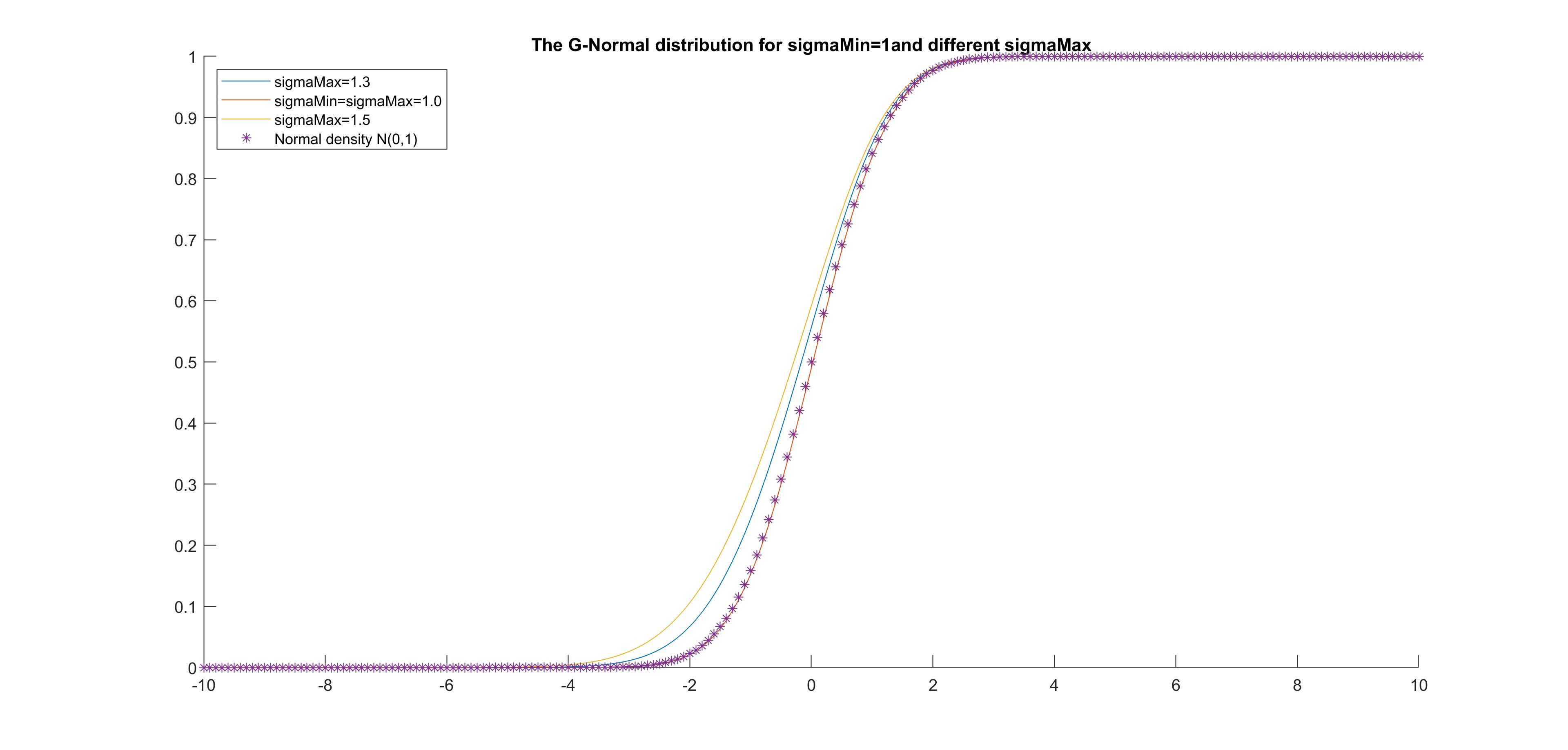}
			\caption{G-Normal distribution} \label{distribution_different_sigmaMax}
	\end{minipage}	
\end{figure}

\begin{figure}[!ht]
\begin{center}
\textbf{Simulation of the G-Normal density and distribution for $%
\sigma_{max}=1.3$ and different $\sigma_{min}$}
\par
\medskip
\end{center}
\par
\begin{minipage}{0.43\textwidth}
		\centering
		\includegraphics[width=1.2\textwidth]{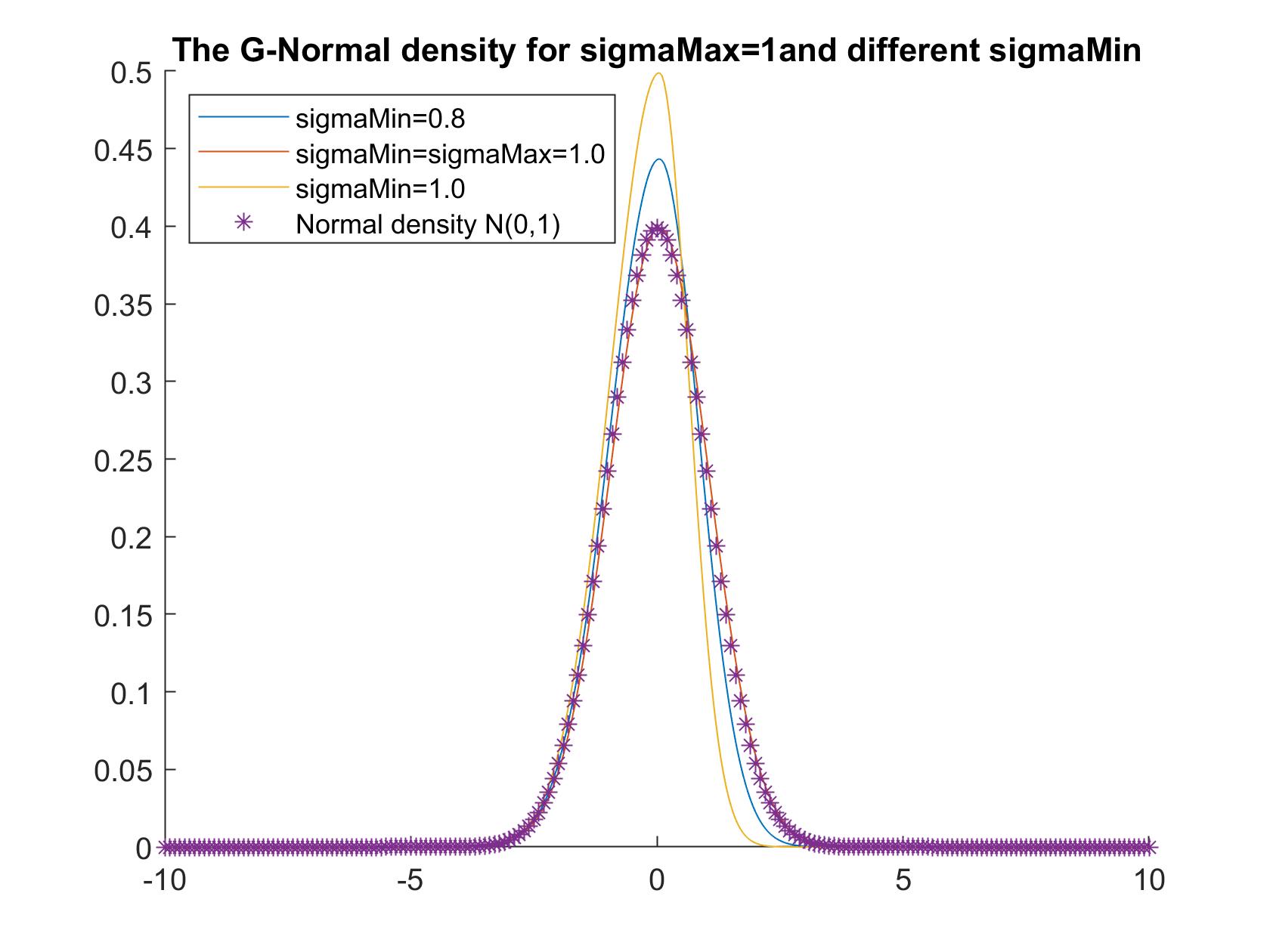}
					\caption{G-Normal density} \label{density_different_sigmaMin}
	\end{minipage}		
\hfill 
\begin{minipage}{0.43\textwidth}
		\centering
		\includegraphics[width=1.85\textwidth]{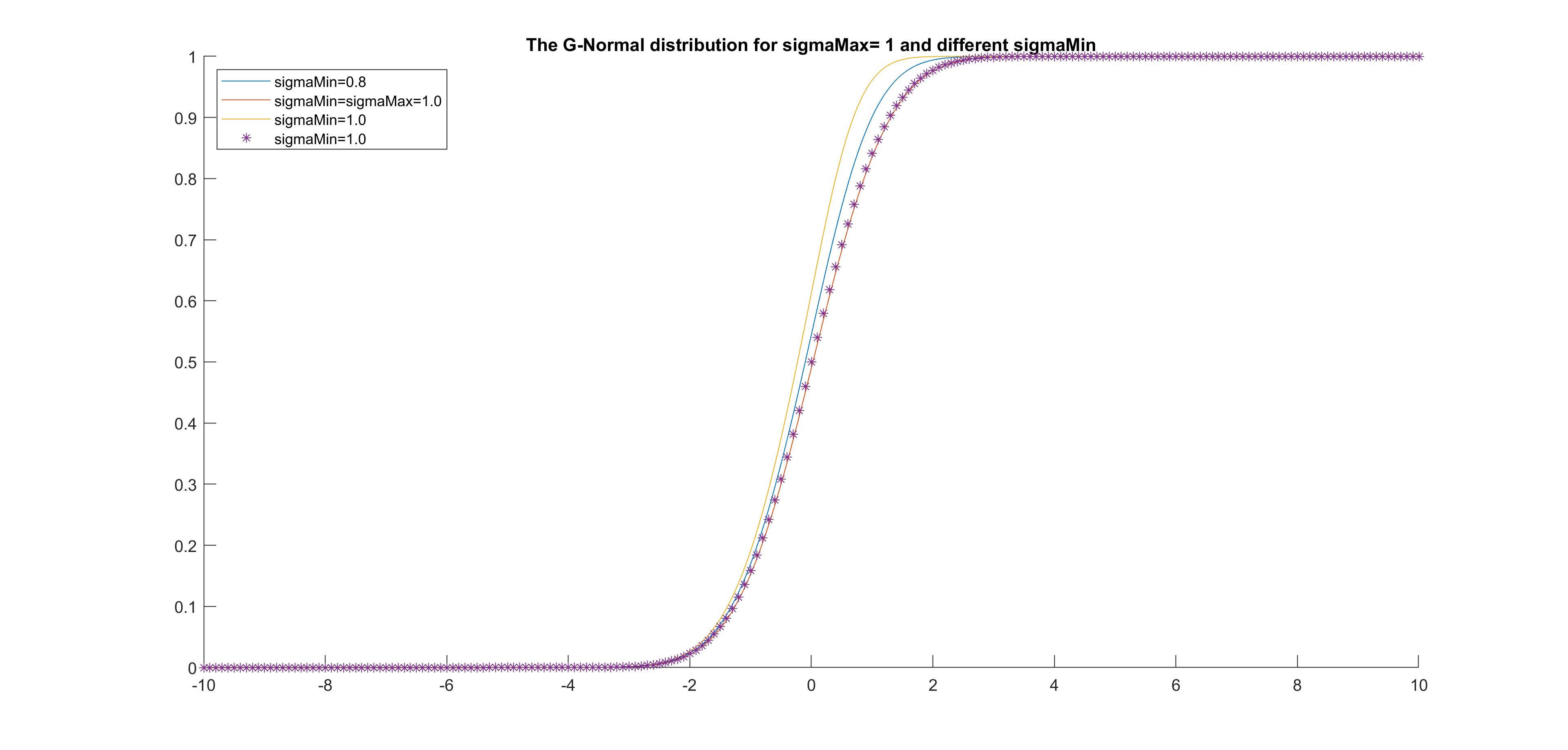}
			\caption{G-Normal distribution} \label{distribution_different_sigmaMin}	
	\end{minipage}	
\end{figure}

Now, let take in this part of this section, $T=1,\tau =0.1$ and the
coefficients of the $G$-NSFDE $(\ref{UCOGNSDE})$ given by: 
\begin{equation*}
Q(t,X_{t}):=0.3\int_{t-\tau }^{t}X(s)ds
\end{equation*}%
\begin{equation*}
b(t,X_{t}):=10\int_{t-\tau }^{t}X(s)ds
\end{equation*}%
\begin{equation*}
\gamma (t,X_{t}):=0.4\int_{t-\tau }^{t}X(s)ds
\end{equation*}%
\begin{equation*}
\sigma (t,X_{t}):=5\int_{t-\tau }^{t}X(s)ds.
\end{equation*}%
For these given data and coefficients we get the following results:

In Figure \ref{G-NSFDE_BM} (resp. Figure \ref{G-NSFDE_BM_Max3}) we
represent the trajectories of the solution of the $G-$NSFDE where the $G$-
Brownian motion is with $\sigma _{max}=1$ (resp. $\sigma _{max}=3$), $\sigma
_{min}=0.65$ and the initial condition $(X_{0}(t))_{-\tau \leq t\leq 0}$
solution of $dX_{0}(t)=dW(t)$ with $X_{0}(-\tau )=0$ where $W$ is the
standard Brownian motion.

\begin{figure}[!ht]
\includegraphics[width=1.0\textwidth]{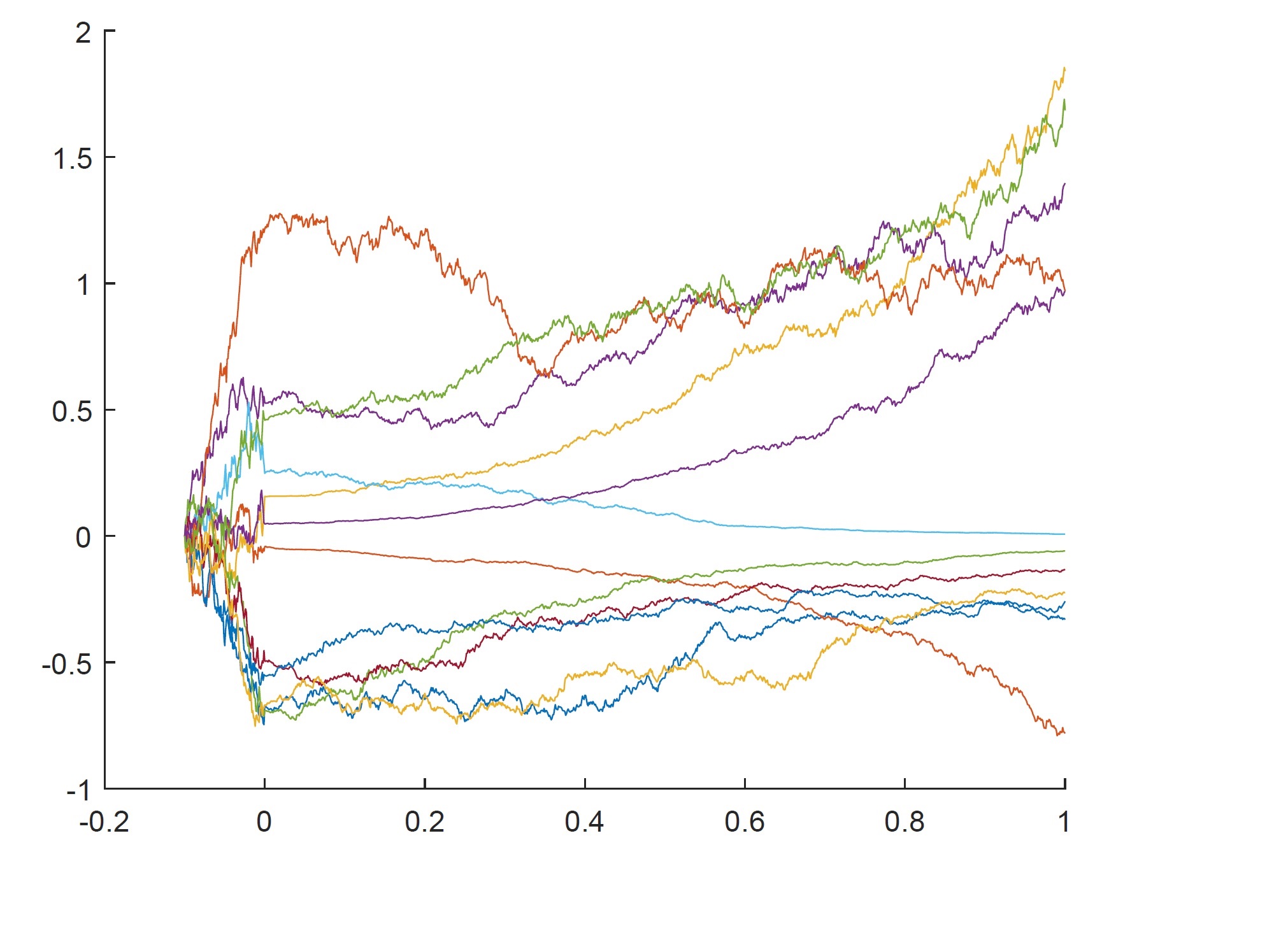}
\caption{Solution G-NSFDE with random initial condition and $\protect\sigma%
_{max}=1, \protect\sigma_{min}=0.65.$}
\label{G-NSFDE_BM}
\end{figure}

\begin{figure}[h]
\includegraphics[width=1.0\textwidth]{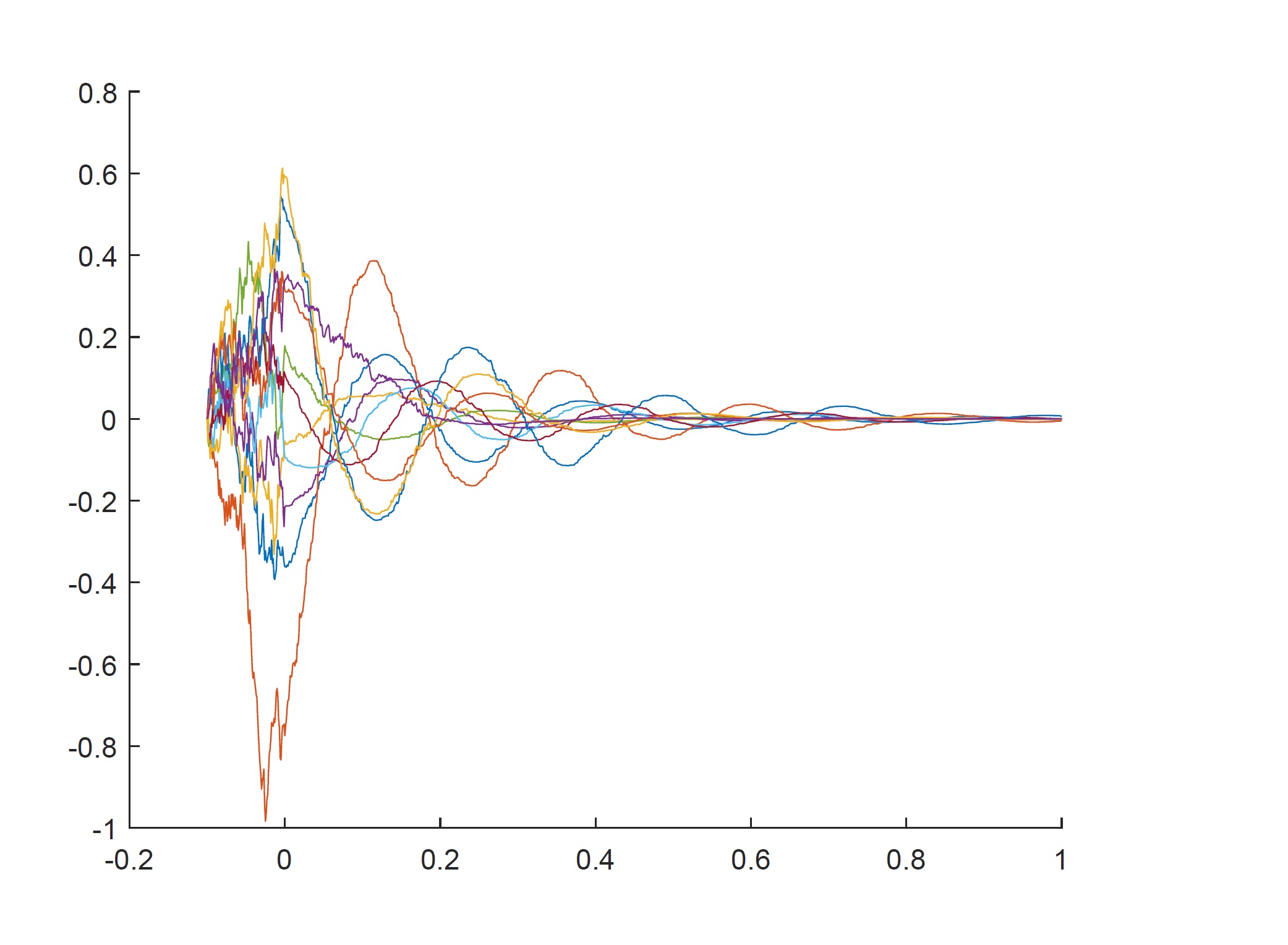}
\caption{Solution G-NSFDE with random initial condition BM and $\protect\sigma%
_{max}=3, \protect\sigma_{min}=0.65.$}
\label{G-NSFDE_BM_Max3}
\end{figure}

In Figure \ref{G-NSFDE_EXP} we represent the trajectories of the solution
of the $G$-NSFDE where the $G$-Brownian motion is with $\sigma
_{max}=1,\sigma _{min}=0.65$ and deterministic initial condition $%
(X_{0}(t))_{-\tau \leq t\leq 0}$ given by $X_{0}(t)=\exp (t)$.

\begin{figure}[h]
\includegraphics[width=1.0\textwidth]{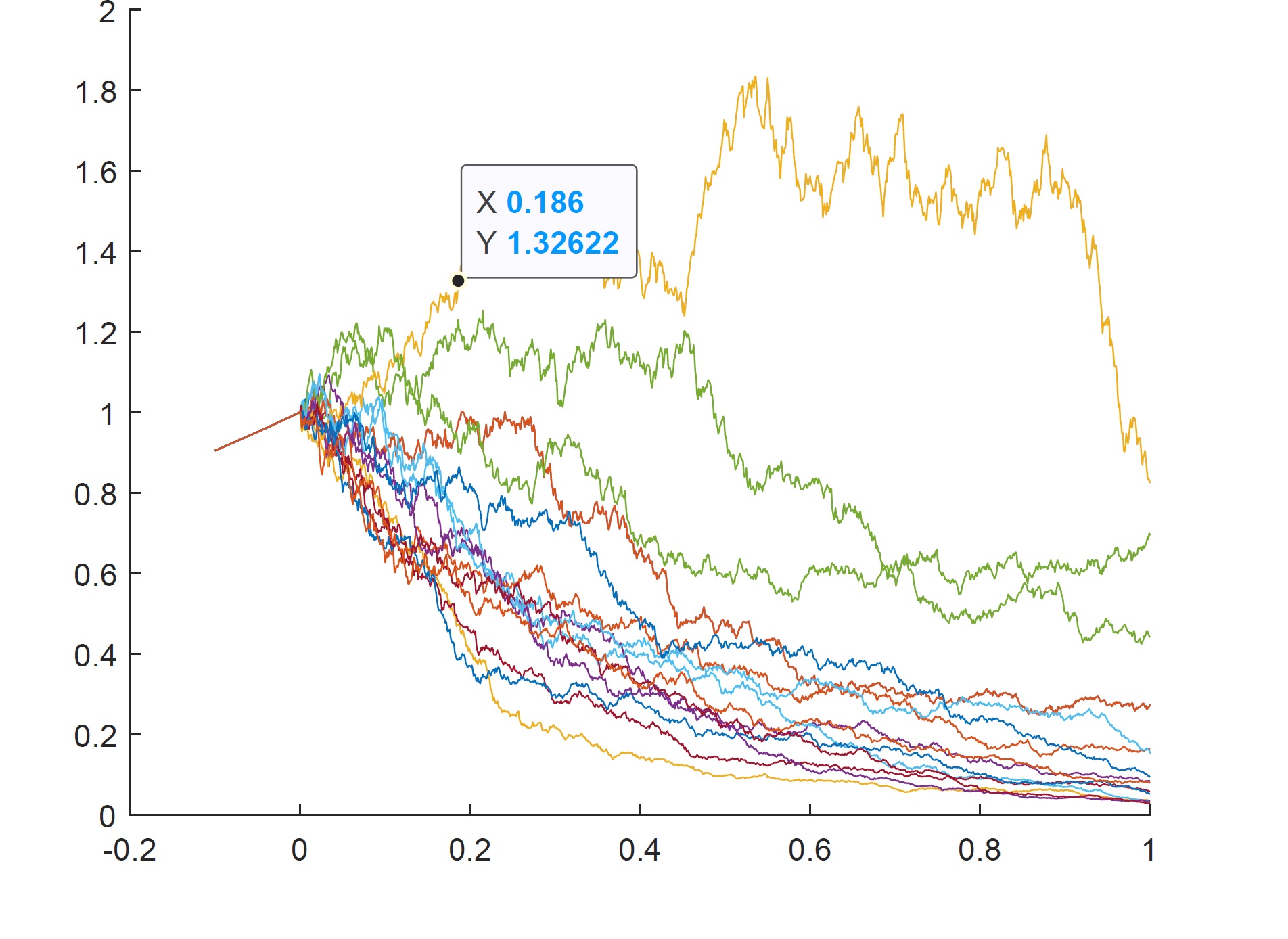}
\caption{Solution G-NSFDE with deterministic initial condition $\exp(t)$ and $%
\protect\sigma_{max}=1, \protect\sigma_{min}=0.65.$}
\label{G-NSFDE_EXP}
\end{figure}

In Figure \ref{G-NSFDE_DETERMINISTIC} we represent the trajectories of the
solution of the $G-$NSFDE where the $G$-Brownian motion is with $\sigma
_{max}=1,\sigma _{min}=0.65$ and deterministic initial condition $%
(X_{0}(t))_{-\tau \leq t\leq 0}$ given by: $\forall t\in \lbrack -\tau
,0],X_{0}(t,\omega )$ a fixed value between $[-0.2,0.2]$.

\begin{figure}[h]
\includegraphics[width=1.0\textwidth]{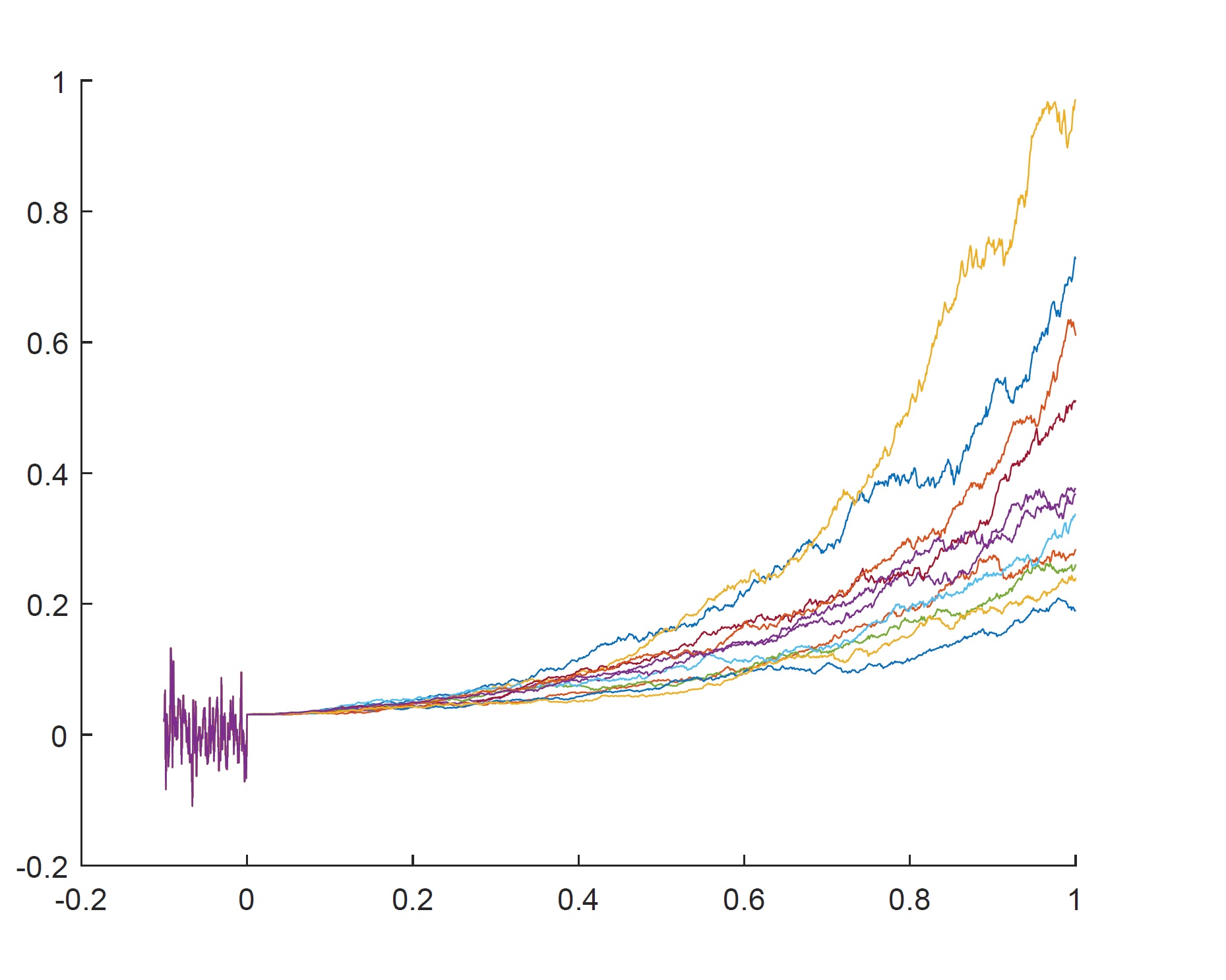}
\caption{Solution G-NSFDE with deterministic initial condition $X_0(t)$ take
values between $[-0.2,0.2]$ for $t\in[-\protect\tau,0]$ and $\protect\sigma
_{max}=1, \protect\sigma_{min}=0.65.$}
\label{G-NSFDE_DETERMINISTIC}
\end{figure}

\clearpage

\bibliographystyle{plain}
\bibliography{biblio}

\begin{thebibliography}{10}

\bibitem{ahmed2014}
N.~Ahmed.
\newblock Stochastic neutral evolution equations on \textsc{H}ilbert spaces
  with partially observed relaxed control and their necessary conditions of
  optimality.
\newblock {\em Nonlinear Analysis: Theory, Methods $\&$ Applications},
  101:66--79, 2014.

\bibitem{bahlali2014}
K.~Bahlali, M.~Mezerdiz, and B.~Mezerdi.
\newblock Existence of optimal controls for systems governed by mean-field
  stochastic differential equations.
\newblock {\em Afrika Statistika}, 9(1):627--645, 2014.

\bibitem{bahlali2006}
S.~Bahlali, B.~Mezerdi, and B.~Djehiche.
\newblock Approximation and optimality necessary conditions in relaxed
  stochastic control problems.
\newblock {\em Journal of Applied Mathematics and Stochastic Analysis}, 2006.

\bibitem{biagini2018}
F.~Biagini, T.~Meyer-Brandis, B.~{$\varnothing$}ksendal, and K.~Paczka.
\newblock Optimal control with delayed information flow of systems driven by
  \textsc{g}-\textsc{b}rownian motion.
\newblock {\em Probability, Uncertainty and Quantitative Risk}, 3(1):1--24,
  2018.

\bibitem{boussinesq1885}
J.~Boussinesq.
\newblock Sur la resistance qu'oppose un fluide indefini en repos, sans
  pesanteur, au mouvement varie d'une sphere solide qu'il mouille sur toute sa
  surface, quand les vitesses restent bien continues et assez faibles pour que
  leurs carres et produits soient negligiables.
\newblock {\em CR Acad. Sc. Paris}, 100:935--937, 1885.

\bibitem{denis2011}
L.~Denis, M.~Hu, and S.~Peng.
\newblock Function spaces and capacity related to a sublinear expectation:
  application to \textsc{g}-\textsc{b}rownian motion paths.
\newblock {\em Potential analysis}, 34(2):139--161, 2011.

\bibitem{denis2006}
L.~Denis and C.~Martini.
\newblock A theoretical framework for the pricing of contingent claims in the
  presence of model uncertainty.
\newblock {\em The Annals of Applied Probability}, 16(2):827--852, 2006.

\bibitem{faizullah2016}
F.~Faizullah.
\newblock Existence results and moment estimates for nsfdes driven by
  \textsc{g}-\textsc{b}rownian motion.
\newblock {\em Journal of Computational and Theoretical Nanoscience},
  13(7):4679--4686, 2016.

\bibitem{faizullah2017}
F.~Faizullah, M.~Bux, M.~Rana, and G.~ur~Rahman.
\newblock Existence and stability of solutions to non-linear neutral stochastic
  functional differential equations in the framework of
  \textsc{g}-\textsc{b}rownian motion.
\newblock {\em Advances in Difference Equations}, 2017(1):1--14, 2017.

\bibitem{gao2009}
F.~Gao.
\newblock Pathwise properties and homeomorphic flows for stochastic
  differential equations driven by \textsc{g}-\textsc{b}rownian motion.
\newblock {\em Stochastic Processes and their Applications},
  119(10):3356--3382, 2009.

\bibitem{hu2014}
M.~Hu, S.~Ji, and S.~Yang.
\newblock A stochastic recursive optimal control problem under the
  \textsc{g}-expectation framework.
\newblock {\em Applied Mathematics $\&$ Optimization}, 70(2):253--278, 2014.

\bibitem{hu2018}
M.~Hu and F.~Wang.
\newblock Stochastic optimal control problem with infinite horizon driven by
  \textsc{g}-\textsc{b}rownian motion.
\newblock {\em ESAIM: Control, Optimisation and Calculus of Variations},
  24(2):873--899, 2018.

\bibitem{peng2007}
S.~Peng.
\newblock G-expectation, \textsc{g}-\textsc{b}rownian motion and related
  stochastic calculus of it{\^o} type.
\newblock In {\em Stochastic analysis and applications}, volume~2, pages
  541--567. 2007.

\bibitem{peng2010}
S.~Peng.
\newblock Nonlinear expectations and stochastic calculus under uncertainty.
\newblock {\em arXiv preprint arXiv:1002.4546}, 24, 2010.

\bibitem{redjil2018}
A.~Redjil and S.~E. Choutri.
\newblock On relaxed stochastic optimal control for stochastic differential
  equations driven by \textsc{g}-\textsc{b}rownian motion.
\newblock {\em ALEA, Lat.Am. J. Probab. Math. Stat}, 15:201–212, 2018.

\bibitem{soner2011m}
H.~M. Soner, N.~Touzi, and J.~Zhang.
\newblock Martingale representation theorem for the \textsc{g}-expectation.
\newblock {\em Stochastic Processes and their Applications}, 121(2):265--287,
  2011.

\bibitem{soner2011}
H.~M. Soner, N.~Touzi, and J.~Zhang.
\newblock Quasi-sure stochastic analysis through aggregation.
\newblock {\em Electronic Journal of Probability}, 16:1844--1879, 2011.

\bibitem{wei2015}
W.~Wei.
\newblock Maximum principle for optimal control of neutral stochastic
  functional differential systems.
\newblock {\em Science China Mathematics}, 58(6):1265--1284, 2015.

\bibitem{yang2016}
J.~Yang and W.~Zhao.
\newblock Numerical simulations for \textsc{g}-\textsc{b}rownian motion.
\newblock {\em Frontiers of Mathematics in China}, 11(6):1625--1643, 2016.

\end{thebibliography}

\end{document}